\documentclass[a4paper]{amsart}
\pdfoutput=1
\usepackage[utf8]{inputenc}
\usepackage[T1]{fontenc}
\usepackage{ae}
\usepackage{aecompl}
\usepackage{lmodern}
\title{Mutation of frozen Jacobian algebras}
\author{Matthew Pressland}
\address{Matthew Pressland\\Institut für Algebra und Zahlentheorie\\Universität Stuttgart\\Pfaffenwaldring 57\\70569 Stuttgart\\Germany}
\email{presslmw@mathematik.uni-stuttgart.de}
\date{\today} 
\usepackage[hmargin=3cm,vmargin=3.5cm]{geometry}
\usepackage{amsmath}
\usepackage{amssymb}
\usepackage{amsfonts}
\usepackage{amsthm}
\usepackage{mathrsfs}
\usepackage{booktabs}
\usepackage[usenames,dvipsnames,svgnames]{xcolor}
\usepackage[colorlinks, urlcolor=Navy, linkcolor=Navy, citecolor=Navy]{hyperref}
\usepackage{microtype}
\usepackage{enumitem}
\usepackage{pdfpages}
\makeatletter
\newcommand{\myitem}[1][]{%
\item[#1]\protected@edef\@currentlabel{#1}\ignorespaces%
}
\makeatother

\usepackage[backend=biber,citestyle=numeric-comp,bibstyle=numeric,maxbibnames=99,firstinits=true,doi=false,isbn=false,url=false,eprint=true]{biblatex}

\renewbibmacro{in:}{%
  \ifentrytype{article}{}{\printtext{\bibstring{in}\intitlepunct}}}

\DeclareFieldFormat[article,inbook,incollection,inproceedings,patent]{title}{#1}
\DeclareFieldFormat[thesis,unpublished]{title}{{\it #1}}

\DeclareFieldFormat[online]{date}{Preprint (#1)}

\DeclareFieldFormat[article]{volume}{\mkbibbold{#1}}
\renewbibmacro*{volume+number+eid}{%
  \printfield{volume}%
  \setunit*{\addnbspace}
  \printfield{number}%
  \setunit{\addcomma\space}%
  \printfield{eid}}
\DeclareFieldFormat[article]{number}{}

\DeclareFieldFormat[book,incollection]{number}{Vol.~#1}
\renewbibmacro*{series+number}{%
  \iffieldundef{series}{}
    {\printfield{series}%
      \iffieldundef{number}{}{\setunit*{\addcomma\addspace}%
       \printfield{number}%
       \newunit}}}

\renewbibmacro*{publisher+location+date}{%
  \printlist{publisher}%
  \setunit*{\addcomma\space}%
  \usebibmacro{date}%
  \newunit}
  
\DeclareFieldFormat{eprint:arxiv}{%
\ifentrytype{online}
  {\ifhyperref
    {\href{http://arxiv.org/abs/#1}{\nolinkurl{arXiv:#1}}}
    {\nolinkurl{arXiv:#1}}
   \iffieldundef{eprintclass}
    {}
    {{\tt \mkbibbrackets{\thefield{eprintclass}}}}}
  {\iffieldundef{eprintclass}
    {\mkbibparens{\ifhyperref
    {\href{http://arxiv.org/abs/#1}{\nolinkurl{arXiv:#1}}}
    {\nolinkurl{arXiv:#1}}}}
    {\mkbibparens{\ifhyperref
    {\href{http://arxiv.org/abs/#1}{\nolinkurl{arXiv:#1}}}
    {\nolinkurl{arXiv:#1}}
    {\tt \mkbibbrackets{\thefield{eprintclass}}}}}}}
\usepackage{color}
\usepackage{tikz}
  \usetikzlibrary{calc}
 \usetikzlibrary{arrows,decorations.markings}
\usepackage{tikz-cd}
\pgfarrowsdeclarecombine{twohead}{twohead}
{angle 90}{angle 90}{angle 90}{angle 90}

\newcommand{\strandcolor}{red}
\newcommand{\graphcolor}{black}
\newcommand{\quivcolor}{black}
\newcommand{\frozcolor}{black}
\newcommand{\bdrycolor}{gray}

\tikzset{strand/.style={\strandcolor,dashed},
 boundary/.style={\bdrycolor, ultra thick}}
\tikzset{bipedge/.style={\graphcolor}}
\tikzset{quivarrow/.style={\quivcolor, -angle 90}}
\tikzset{frozarrow/.style={\frozcolor, -angle 90, dashed}}

\newcommand{\bstart}{130} 
\newcommand{\fifth}{72} 
\newcommand{\dotrad}{1.3pt} 

\usepackage[capitalise,compress]{cleveref}
\Crefname{eg}{Example}{Examples}
\Crefname{thm}{Theorem}{Theorems}
\Crefname{conj}{Conjecture}{Conjectures}
\Crefname{prop}{Proposition}{Propositions}
\Crefname{lem}{Lemma}{Lemmas}
\Crefname{defn}{Definition}{Definitions}
\Crefname{thm*}{Theorem}{Theorems}
\Crefname{conj*}{Conjecture}{Conjectures}
\Crefname{ques}{Question}{Questions}
\Crefname{rem}{Remark}{Remarks}

\numberwithin{equation}{section}

\usepackage{todonotes}

\theoremstyle{plain}
\newtheorem{thm}{Theorem}[section]
\newtheorem{thm*}{Theorem}
\newtheorem{lem}[thm]{Lemma}
\newtheorem{cor}[thm]{Corollary}
\newtheorem{prop}[thm]{Proposition}

\theoremstyle{definition}
\newtheorem{defn}[thm]{Definition}
\newtheorem{eg}[thm]{Example}

\newtheorem{rem}[thm]{Remark}
\theoremstyle{remark}

\newcommand{\NN}{\mathbb{N}}
\newcommand{\CC}{\mathbb{C}}
\newcommand{\KK}{\mathbb{K}}

\newcommand{\cat}{\mathcal{C}}
\newcommand{\subcat}{\mathcal{D}}
\newcommand{\frobcat}{\mathcal{E}}

\DeclareMathOperator{\Sub}{Sub}
\DeclareMathOperator{\CM}{CM}

\DeclareMathOperator{\fgmod}{mod}

\DeclareMathOperator{\proj}{proj}
\DeclareMathOperator{\Ext}{Ext}
\DeclareMathOperator{\Hom}{Hom}
\DeclareMathOperator{\stabEnd}{\underline{End}}

\DeclareMathOperator{\End}{End}
\DeclareMathOperator{\add}{add}
\DeclareMathOperator{\GP}{GP}

\DeclareMathOperator{\gldim}{gl.dim}

\DeclareMathOperator{\catrad}{Rad}
\DeclareMathOperator{\Tr}{Tr}
\DeclareMathOperator{\Def}{Def}
\DeclareMathOperator{\Char}{char}

\newcommand{\cpa}[2]{#1\langle\hspace{-0.1em}\langle #2\rangle\hspace{-0.1em}\rangle}

\newcommand{\compgen}[2]{#1\langle\hspace{-0.1em}\langle #2\rangle\hspace{-0.1em}\rangle}

\newcommand{\FZ}{\mathrm{FZ}}
\newcommand{\op}{\mathrm{op}}
\newcommand{\red}{\mathrm{red}}
\newcommand{\triv}{\mathrm{triv}}

\newcommand{\iso}{\cong}
\newcommand{\isoto}{\stackrel{\sim}{\to}}
\newcommand{\tens}{\mathbin{\otimes}}
\newcommand{\comptens}{\mathbin{\widehat{\otimes}}}
\newcommand{\dsum}{\mathbin{\oplus}}
\newcommand{\bigdsum}{\bigoplus}
\newcommand{\union}{\cup}
\newcommand{\id}[1]{\mathrm{id}_{#1}}

\newcommand{\powser}[2]{#1[\hspace{-0.1em}[#2]\hspace{-0.1em}]}

\newcommand{\Endalg}[2]{\End_{#1}(#2)^{\op}}
\newcommand{\stabEndalg}[2]{\stabEnd_{#1}(#2)^{\op}}

\newcommand{\head}[1]{h#1}
\newcommand{\tail}[1]{t#1}

\newcommand{\jac}[2]{\mathcal{J}(#1,#2)}
\newcommand{\frjac}[3]{\mathcal{J}(#1,#2,#3)}

\newcommand{\lift}[1]{\widetilde{#1}}

\newcommand{\Kdual}{\mathrm{D}}

\newcommand{\stab}[1]{\underline{#1}}

\newcommand{\set}[1]{\left\{#1\right\}}
\newcommand{\Span}[1]{\langle#1\rangle}

\newcommand{\der}[1]{\partial_{#1}}
\newcommand{\leftder}[1]{\partial^l_{#1}}
\newcommand{\rightder}[1]{\partial^r_{#1}}

\newcommand{\close}[1]{\overline{#1}}

\newcommand{\frozen}{dashed}

\newcommand{\mut}{\mathrm{m}}

\begin{document}

\begin{abstract}
We survey results on mutations of Jacobian algebras, while simultaneously extending them to the more general setup of frozen Jacobian algebras, which arise naturally from dimer models with boundary and in the context of the additive categorification of cluster algebras with frozen variables via Frobenius categories. As an application, we show that the mutation of cluster-tilting objects in various such categorifications, such as the Grassmannian cluster categories of Jensen--King--Su, is compatible with Fomin--Zelevinsky mutation of quivers. We also describe an extension of this combinatorial mutation rule allowing for arrows between frozen vertices, which the quivers arising from categorifications and dimer models typically have.
\end{abstract}
\maketitle

\section{Introduction}

Jacobian algebras, defined via the data of a quiver with potential, play an important role in the theory of cluster algebras, particularly in the context of their categorification by triangulated categories \cite{buanmutation,kulkarnidimer,amiotcluster,derksenquivers1}. However, the concept predates this subject, appearing for example in the mathematical physics of dimer models \cite{hananydimer}, which has then found applications in algebraic and noncommutative geometry \cite{broomheaddimer, davisonconsistency, bocklandtconsistency, bocklandtgeometric} and mirror symmetry \cite{bocklandtnoncommutative}. More recently, it has been fruitful to replace the quiver by an ice quiver, by declaring a particular subquiver to be frozen, leading to the more general notion of a frozen Jacobian algebra. These algebras appear naturally when considering dimer models on surfaces with boundary \cite{francobipartite}, as well as endomorphism algebras of cluster-tilting objects in Frobenius categorifications of cluster algebras with frozen variables \cite{baurdimer,buanmutation,presslandcategorification}.

The goal of the present paper is to fill a literature gap by extending several key results about ordinary Jacobian algebras to the frozen case. In outline, our results are as follows.

\begin{enumerate}[label=(\roman*)]
\item\label{contents-reduction} In \Cref{s:reduction}, we explain how to modify an ice quiver with potential, without changing the isomorphism class of the frozen Jacobian algebra it defines, so that the quiver is the Gabriel quiver of this algebra.
\item\label{contents-mutation} In \Cref{s:mutation}, we explain how ice quivers with potential are transformed under a local move called mutation, and conditions on the potential which make this process compatible with Fomin--Zelevinsky mutation of quivers. Fomin--Zelevinsky mutation is usually defined for ice quivers not having any arrows between their frozen vertices; since we want to allow such arrows, we extend the definition accordingly. While this extension is likely already known to experts, we are not aware of it having been previously formalised in the literature.
\item\label{contents-ct-objects} In \Cref{s:ct-objects}, we consider frozen Jacobian algebras arising as endomorphism algebras of cluster-tilting objects in Frobenius cluster categories \cite[Defn.~3.3]{presslandinternally}, which are certain stably $2$-Calabi--Yau Frobenius categories well-suited to the additive categorification of cluster algebras and appear frequently in the literature \cite{buanmutation,geisspartial,
jensencategorification,presslandcategorification}. Such an algebra can be mutated in two ways; firstly as a frozen Jacobian algebra, using the combinatorial mutation procedure of \Cref{s:mutation}, and secondly by mutating the cluster-tilting object in the sense of Iyama--Yoshino \cite{iyamamutation} and taking the endomorphism algebra of the result. We give sufficient conditions on the category for these two operations to coincide, and for them to induce an extended Fomin--Zelevinsky mutation of the Gabriel quiver of the endomorphism algebra. This leads to new results on several well-studied classes of Frobenius cluster category, such as the Grassmannian cluster categories of Jensen--King--Su \cite{jensencategorification}.
\end{enumerate}

We stress that this paper owes a significant debt to the work preceding it, particularly that of Derksen--Weyman--Zelevinsky \cite{derksenquivers1} (who deal with points \ref{contents-reduction} and \ref{contents-mutation} in the case of ordinary Jacobian algebras) and Buan--Iyama--Reiten--Smith \cite{buanmutation} (who deal with \ref{contents-ct-objects} for ordinary Jacobian algebras and triangulated categories), and for this reason we consider the present paper to be in part a survey of their results. At the same time, the generalisations we present here are by now applicable enough that we felt it necessary to make them explicit in the literature---the lack of a clean reference for these statements may be partly responsible for the fact that papers providing Frobenius categorifications of cluster algebras with frozen variables have often avoided commenting explicitly on compatibility of mutations in the sense of \ref{contents-ct-objects}.

We note that an earlier version of \Cref{s:ct-objects} was included in the first arXiv version of \cite{presslandcategorification}, where some of the results are applied.

Throughout, all algebras are $\KK$-algebras, and all categories $\KK$-categories, over a field $\KK$. Without further explanation, `module' is taken to mean `left module'. Arrows and functions are composed from right to left.

\section{Frozen Jacobian algebras}
\label{s:FJAs}

We begin by introducing the various pieces of combinatorial data needed to define a frozen Jacobian algebra.

\begin{defn}
A \emph{quiver} is a tuple $Q=(Q_0,Q_1,\head,\tail)$, where $Q_0$ and $Q_1$ are sets, and $\head,\tail\colon Q_1\to Q_0$ are functions. Graphically, we think of the elements of $Q_0$ as vertices and those of $Q_1$ as arrows, so that each $\alpha\in Q_1$ is realised as an arrow $\alpha\colon\tail{\alpha}\to\head{\alpha}$. We call $Q$ \emph{finite} if $Q_0$ and $Q_1$ are finite sets.
\end{defn}

\begin{defn}
Let $Q$ be a quiver. A quiver $F=(F_0,F_1,\head',\tail')$ is a \emph{subquiver} of $Q$ if it is a quiver such that $F_0\subseteq Q_0$, $F_1\subseteq Q_1$ and the functions $\head'$ and $\tail'$ are the restrictions of $\head$ and $\tail$ to $F_1$. We say $F$ is a \emph{full} subquiver if $F_1=\{\alpha\in Q_1:\head{\alpha},\tail{\alpha}\in F_0\}$, so that a full subquiver of $Q$ is completely determined by its set of vertices. 
\end{defn}

\begin{defn}
An \emph{ice quiver} is a pair $(Q,F)$, where $Q$ is a quiver, and $F$ is a (not necessarily full) subquiver of $Q$. We call $F_0$, $F_1$ and $F$ the \emph{frozen} vertices, arrows and subquiver respectively. Vertices of $Q$ not in $F_0$ will be called \emph{mutable} (cf.~Definition~\ref{iqpmutation}), whereas arrows of $Q$ not in $F_1$ will be simply called \emph{unfrozen}.
\end{defn}

\begin{rem}
We note that by not insisting that $F$ is a full subquiver in the above definition, our subsequent definitions will differ slightly from those of other authors, e.g.\ \cite[Defn.~1.1]{buanmutation}, who do make this insistence. Aside from being more general, allowing for the case that $F$ is not a full subquiver of $Q$ is convenient when it comes to describing mutations of ice quivers with potential in Section~\ref{s:mutation}.
\end{rem}

We also note that an ice quiver is part of the data of a seed in a cluster algebra; we recommend Keller's survey \cite{kellercluster} for an overview of this construction. In this context, while the choice of frozen vertices is important, since they correspond to the frozen variables appearing in every cluster, frozen arrows (or indeed any arrows between the frozen vertices) play no role, and so are often omitted. However, in the context of categorification of cluster algebras, which we will discuss more in Section~\ref{s:ct-objects}, these arrows become relevant again; the categorification enhances the data of a seed by replacing the ice quiver by an algebra, and the Gabriel quiver of this algebra typically does have arrows between its frozen vertices. Another situation in which such arrows appear naturally is in the context of dimer models on surfaces with boundary, as we now explain.

\begin{eg}
\label{e:dimers1}
A \emph{dimer model} is a pair $(\Sigma,G)$ in which $\Sigma$ is an oriented surface, possibly with boundary, and $G$ is a bipartite graph embedded in $\Sigma$ such that $\partial\Sigma$ does not intersect the vertex set of the graph, together with a collection of arcs in $\Sigma$ which intersect the graph in a vertex and $\partial\Sigma$ in one point, such that $\Sigma\setminus G$ is a disjoint union of discs. We usually think of $G$ simply as a graph, with the additional arcs being \emph{half-edges} that have one end-point `outside' of $\Sigma$. Indeed, if $\Sigma$ is a closed surface, then $G$ is an honest graph, and any dimer model $(\Sigma,G)$ can be realised as $(\Sigma,\Sigma\cap\widehat{G})$ for some dimer model $(\widehat{\Sigma},\widehat{G})$, where $\widehat{\Sigma}$ is a closed surface containing $\Sigma$. A concrete example of a dimer model drawn in the disk is shown in Figure~\ref{f:dimereg}.

\begin{figure}
    \centering
    \begin{minipage}[t]{0.45\textwidth}
        \centering
\begin{tikzpicture}[scale=2.5,baseline=(bb.base)]

\path (0,0) node (bb) {};

\draw [boundary] (0,0) circle(1.0);

\foreach \n/\a in {1/-12, 2/0, 3/5, 4/10, 5/12}
{ \coordinate (b\n) at (\bstart-\fifth*\n+\a:1.0);
\coordinate (l\n) at (\bstart-\fifth*\n+\a-4:0.85);
\coordinate (r\n) at (\bstart-\fifth*\n+\a+4:0.85);}

\foreach \n/\m in {6/1, 7/2, 8/3, 9/4, 10/5}
  {\coordinate (b\n) at ($0.55*(b\m)$);}
\foreach \n/\m in {12/2, 13/4}
  {\coordinate (b\n) at ($0.80*(b\m)$);}

\coordinate (b11) at ($0.4*(b6) + 0.4*(b10)$);

\foreach \n/\s in {6/1.2, 7/0.9, 8/0.7, 9/0.9, 10/1.2}
{\coordinate (b\n) at ($\s*(b\n)$);}

\foreach \h/\t in {1/6, 2/12, 7/12, 3/8, 4/13, 9/13, 5/10, 6/7, 7/8, 8/9, 9/10, 10/11, 6/11, 8/11}
{ \draw [bipedge] (b\h)--(b\t); }

\foreach \n in {6,8,10,12,13} 
  {\draw [\graphcolor] (b\n) circle(\dotrad) [fill=\graphcolor];} 
\foreach \n in {7,9,11}  
  {\draw [\graphcolor] (b\n) circle(\dotrad) [fill=white];} 

\foreach \e/\f/\t in {1/6/0.5, 2/12/0.5, 7/12/0.5, 3/8/0.5, 4/13/0.5, 9/13/0.5, 5/10/0.5, 6/7/0.5, 7/8/0.5, 8/9/0.5, 9/10/0.5, 10/11/0.5, 6/11/0.5, 8/11/0.6}
{\coordinate (a\e-\f) at ($(b\e) ! \t ! (b\f)$); }

 \end{tikzpicture}

\caption{A dimer model in the disk.}
\label{f:dimereg}
    \end{minipage}\hfill
    \begin{minipage}[t]{0.55\textwidth}
        \centering
\begin{tikzpicture}[scale=2.5,baseline=(bb.base)]

\path (0,0) node (bb) {};


\foreach \n/\a in {1/-12, 2/0, 3/5, 4/10, 5/12}
{ \coordinate (b\n) at (\bstart-\fifth*\n+\a:1.0);
  \coordinate (l\n) at (\bstart-\fifth*\n+\a-4:0.85);
  \coordinate (r\n) at (\bstart-\fifth*\n+\a+4:0.85);}

\foreach \n/\m in {6/1, 7/2, 8/3, 9/4, 10/5}
  {\coordinate (b\n) at ($0.55*(b\m)$);}
\foreach \n/\m in {12/2, 13/4}
  {\coordinate (b\n) at ($0.80*(b\m)$);}

\coordinate (b11) at ($0.4*(b6) + 0.4*(b10)$);

\foreach \n/\s in {6/1.2, 7/0.9, 8/0.7, 9/0.9, 10/1.2}
{\coordinate (b\n) at ($\s*(b\n)$);}



\foreach \e/\f/\t in {1/6/0.5, 2/12/0.5, 7/12/0.5, 3/8/0.5, 4/13/0.5, 9/13/0.5, 5/10/0.5, 6/7/0.5, 7/8/0.5, 8/9/0.5, 9/10/0.5, 10/11/0.5, 6/11/0.5, 8/11/0.6}
{\coordinate (a\e-\f) at ($(b\e) ! \t ! (b\f)$); }

\foreach \n/\m/\a in {1/45/0, 2/15/-7, 3/12/0, 4/23/10, 5/34/12}
{ \draw (\bstart+\fifth/2-\fifth*\n+\a:0.93) node (q\m) {$\diamond$}; }

\foreach \m/\a/\r in {25/14/0.25, 24/176/0.22}
{ \draw (\a:\r) node (q\m) {$\bullet$}; }

\foreach \t/\h/\a in {23/12/-27, 34/23/-27, 45/34/-32, 15/45/-32, 12/15/-27}
{ \draw [frozarrow]  (q\t) edge [bend left=\a] (q\h); }
 
\foreach \t/\h/\a in {25/15/12, 15/12/-25, 12/25/23, 25/24/2, 24/45/20, 45/25/20, 24/23/20, 23/34/-20, 34/24/12}
{ \draw [quivarrow]  (q\t) edge [bend left=\a] (q\h); }

 \end{tikzpicture}

\caption{The ice quiver of the dimer model shown in Figure~\ref{f:dimereg}. Frozen arrows are dashed, and frozen vertices shown as white diamonds.}
\label{f:dimerquiveg}
    \end{minipage}
\end{figure}

A dimer model $(\Sigma,G)$ determines an ice quiver $(Q,F)$, also embedded in $\Sigma$, as follows. The vertices of $Q$ are the connected components of $\Sigma\setminus G$, which we call the \emph{faces} of the dimer, and each edge (or half-edge) of $G$ determines an arrow between the faces it separates; we think of the arrow and this edge as dual, and orient the arrow so that the black vertex of the dual edge is on its left. The vertices of $F$ are the components of $\Sigma\setminus G$ meeting $\partial\Sigma$, and its arrows are those arrows of $Q$ dual to a half-edge of $G$. An explicit example is shown in Figure~\ref{f:dimerquiveg}.

Each vertex of $G$ determines a cycle in $Q$, by composing the arrows dual to edges incident with the vertex. This cycle is oriented in a way consistent with the orientation of the surface when the vertex is black, and with the opposite orientation when the vertex is white.

Originally appearing in the context of statistical mechanics \cite{kasteleynstatistics,temperleydimer}, these constructions have been well-studied in the mathematics and physics literature; see for example \cite{okounkovquantum,hananydimer,broomheaddimer,davisonconsistency} in the case that $\Sigma$ is closed, and \cite{baurdimer,francobipartite} in the general case.
\end{eg}

\begin{defn}
Let $Q$ be a quiver, and let $S$ be the semisimple $\KK$-algebra whose underlying vector space has the basis $\set{e_v:v\in Q_0}$, with multiplication induced from $e_v\cdot e_w=\delta_{vw}e_v$. Equip $\KK Q_1$ with the structure of an $S$-bimodule by defining
\[e_{\head{\alpha}}\alpha e_{\tail{\alpha}}=\alpha,\]
noting that this implies $e_v\alpha e_w=0$ whenever $v\ne\head{\alpha}$ or $w\ne\tail{\alpha}$. Then the \emph{complete path algebra} of $Q$ is the complete tensor algebra
\[\cpa{\KK}{Q}:=\compgen{S}{\KK Q_1},\]
so $\cpa{\KK}{Q}$ has underlying vector space
\[\prod_{d=0}^\infty(\KK Q_1)^{\tens_Sd},\]
and multiplication induced from the tensor product. In practice, the elements of $\cpa{\KK}{Q}$ are possibly infinite $\KK$-linear combinations of paths of $Q$, and the product of two paths is their concatenation when this is well-defined, and zero when it is not. We treat $\cpa{\KK}{Q}$ as a topological algebra by equipping it with the $J$-adic topology, where $J$ is the two-sided \emph{arrow ideal}
\[J=\prod_{d=1}^\infty(\KK Q_1)^{\tens_Sd}.\]
This allows us to talk about closed ideals in $\cpa{\KK}{Q}$; the closure of the ideal generated by a set $R\subseteq\cpa{\KK}{Q}$ is
\[\close{\Span{R}}=\set{\sum_{i=1}^\infty a_ir_ib_i:\text{$a_i,b_i\in\cpa{\KK}{Q}$, $r_i\in R$}}.\]
We may identify $S$ with the subalgebra of $\cpa{\KK}{Q}$ spanned by the length-zero paths, i.e.\ with $\KK Q_0$, by associating each vertex $v$ with the generator $e_v$ of $S$.
\end{defn}

\begin{eg}
\label{e:cpas}
The reader is warned that, when $Q$ has oriented cycles, the complete path algebra $\cpa{\KK}{Q}$ can be rather different from the ordinary (uncompleted) path algebra $\KK Q$, whose elements are finite linear combinations of paths. The key fact is that idempotents can be lifted from $S\cong\cpa{\KK}{Q}/J$ to $\cpa{\KK}{Q}$ (cf.\ \cite[Lem.~I.4.4]{assemelements}), meaning that many techniques used in the representation theory of finite-dimensional algebras apply equally well to the (often infinite-dimensional) algebra $\cpa{\KK}{Q}$.

An instructive example is to consider the case that $Q$ consists of a single vertex and a loop $x$, so that $\KK Q=\KK[x]$ is a polynomial ring in one variable, whereas $\cpa{\KK}{Q}=\powser{\KK}{x}$ is a ring of power series. Thus $\KK Q$ has maximal ideals $\Span{x-\lambda}$ for each $\lambda\in\KK$, so it has one simple module for each element of $\KK$ and its Jacobson radical is zero. On the other hand, $\cpa{\KK}{Q}$ is in many ways more like a finite-dimensional path algebra---its unique maximal ideal is $J=\compgen{}{x}$, generated by the arrow, so the arrow ideal is the Jacobson radical, and there is a unique simple module, up to isomorphism, corresponding to the unique vertex of $Q$.
\end{eg}

\begin{defn}
\label{d:potential}
Let $Q=(Q_0,Q_1,\head,\tail)$ be a finite quiver. We may grade $\cpa{\KK}{Q}$ by path length; since $\cpa{\KK}{Q}$ is generated by vertices and arrows, which are homogeneous, we get an induced grading on the vector space quotient $\cpa{\KK}{Q}/\close{\{\cpa{\KK}{Q},\cpa{\KK}{Q}\}}$, where  $\{\cpa{\KK}{Q},\cpa{\KK}{Q}\}$ denotes the vector subspace of $\cpa{\KK}{Q}$ spanned by commutators\footnote{This vector subspace is mistakenly replaced by the ideal generated by commutators in the published version.}. Then a \emph{potential} on $Q$ is an element $W\in\cpa{\KK}{Q}/\close{\{\cpa{\KK}{Q},\cpa{\KK}{Q}\}}$ expressible as a (possibly infinite) linear combination of homogeneous elements of degree at least $2$, such that any term involving a loop has degree at least $3$. An \emph{ice quiver with potential} is a tuple $(Q,F,W)$ in which $(Q,F)$ is a finite ice quiver, and $W$ is a potential on $Q$. If $F=\varnothing$ is the empty quiver, then $(Q,\varnothing,W)=:(Q,W)$ is called simply a \emph{quiver with potential}.
\end{defn}

\begin{rem}
While the definition in \cite{derksenquivers1} of a quiver with potential does not allow any loops in the quiver $Q$, our weaker assumption controlling how they appear in the potential will be sufficient for some purposes; see Section~\ref{s:reduction} for some comments on how this assumption is used. We want to allow this increased level of generality where possible, since quivers with potential including loops appear naturally in certain contexts, e.g.\ as contraction algebras in the sense of Donovan and Wemyss \cite{donovannoncommutative}. When we move on to discussing mutations in Sections~\ref{s:mutation} and \ref{s:ct-objects}, we will need stronger assumptions on the non-existence of loops.
\end{rem}

One may think of a potential as a formal linear combination of cyclic paths in $Q$ (of length at least $2$), considered up to the equivalence relation on such cycles induced by
\[\alpha_n\dotsm\alpha_1\sim\alpha_{n-1}\dotsm\alpha_1\alpha_n,\]
since every element $\cpa{\KK}{Q}/\close{\{\cpa{\KK}{Q},\cpa{\KK}{Q}\}}$ is uniquely expressible as a (possibly infinite) linear combination of equivalence classes of cycles under this relation.

The combinatorial data of an ice quiver with potential can be used to define an algebra, which is our main object of study.

\begin{defn}
\label{d:jac-alg}
Let $p=\alpha_n\dotsm\alpha_1$ be a cyclic path, with each $\alpha_i\in Q_1$, and let $\alpha\in Q_1$ be any arrow. Then the \emph{cyclic derivative} of $p$ with respect to $\alpha$ is
\[\der{\alpha}{p}:=\sum_{\alpha_i=\alpha}\alpha_{i-1}\dotsm\alpha_1\alpha_n\dotsm\alpha_{i+1}.\]
Extending $\der{\alpha}{}$ by linearity and continuity, it determines a map $\cpa{\KK}{Q}/\close{\{\cpa{\KK}{Q},\cpa{\KK}{Q}\}}\to\cpa{\KK}{Q}$.
For an ice quiver with potential $(Q,F,W)$, we define the \emph{frozen Jacobian algebra}
\[\frjac{Q}{F}{W}=\cpa{\KK}{Q}/\close{\Span{\der{\alpha}W:\alpha\in Q_1\setminus F_1}}.\]
If $F=\varnothing$, we omit it from the notation, and call $\jac{Q}{W}:=\frjac{Q}{\varnothing}{W}$ the \emph{Jacobian algebra} of the quiver with potential $(Q,W)$.
\end{defn}

\begin{rem}
To compute the cyclic derivatives $\partial_\alpha W$, we pick a representative of $W$ in $\cpa{\KK}{Q}$. It is straightforward to check that the result is independent of this choice. Note that not all of the data in the frozen subquiver $F$ is used in the definition; rather, we need only the set $F_1$ of arrows. Said differently, any vertices of $F_0$ not incident with any arrows in $F_1$ can be freely chosen to be mutable or frozen without affecting the algebra $\frjac{Q}{F}{W}$. However, recording all the frozen vertices is more compatible with the defining data of a cluster algebra with frozen variables, and the choice of which vertices are frozen plays a role in some of our later results, primarily in Section~\ref{s:ct-objects}. The choice of frozen vertices---or more precisely the sum of vertex idempotents at these vertices---is also important in \cite[\S5]{presslandinternally}.
\end{rem}

Jacobian algebras are somewhat ubiquitous; it has been shown by Buan--Iyama--Reiten--Smith \cite[Cor.~6.8]{buanmutation} (see also Keller \cite[Thm.~6.12]{kellerdeformed}) that cluster-tilted algebras are finite-dimensional Jacobian algebras, and Bocklandt \cite[Thm.~3.1]{bocklandtgraded} has shown that any graded $3$-Calabi--Yau algebra is a (necessarily infinite-dimensional) Jacobian algebra. The author \cite[\S5]{presslandinternally} has shown that frozen Jacobian algebras are good candidates for internally $3$-Calabi--Yau algebras (defined in loc.\ cit.) which, under some additional `smallness' conditions, can be used to construct Frobenius cluster categories \cite[Thm.~4.1]{presslandinternally}.

\begin{eg}
\label{e:dimers2}
Let $(\Sigma,G)$ be a dimer model, as in Example~\ref{e:dimers1}, defining an ice quiver $(Q,F)$. We already noted that each vertex $v$ of $G$ determines a cycle $C_v$ of $Q$, with orientation (relative to that of $\Sigma$) determined by the colour of $v$. We use these cycles to define a potential on $Q$, by
\[W=\sum_{\text{$v$ black}}C_v-\sum_{\text{$v$ white}}C_v.\]
The Jacobian algebra $\frjac{Q}{F}{W}$ is called the \emph{dimer algebra} of $(\Sigma,G)$. The relations arising from the potential are sometimes known as \emph{F-term} relations, and can be described as follows: each unfrozen arrow $\alpha$ can be completed in two ways to a cycle around one of the vertices of $G$, one winding around a black vertex and the other around a white vertex, and the relation $\partial_\alpha W=0$ means that the two paths obtained by removing $\alpha$ from these cycles are equal in the Jacobian algebra. When $\alpha$ is a frozen arrow, dual to a half-edge, only one of the two cycles will exist, and the path obtained by removing $\alpha$ from this cycle defines a non-zero element of $\frjac{Q}{F}{W}$.

In the case that $\Sigma$ is closed, Broomhead \cite{broomheaddimer} showed that under various consistency conditions on $G$ (the strongest of which implies that $\Sigma\iso S^1\times S^1$ is a torus), the dimer algebra is a $3$-Calabi--Yau noncommutative crepant resolution of a toric singularity. In the case that $\Sigma$ is a disk, Baur--King--Marsh \cite{baurdimer} show, for careful choices of $G$, that the dimer algebra is the endomorphism algebra of a cluster-tilting object in Jensen--King--Su's Grassmannian cluster category \cite{jensencategorification}.
\end{eg}

\section{Reduction}
\label{s:reduction}

In this section, we discuss some operations on an ice quiver with potential $(Q,F,W)$ that do not affect the isomorphism class of $\frjac{Q}{F}{W}$. Our first reduction is straightforward, and only applies when the subquiver $F$ has cycles.

\begin{defn}
Let $(Q,F,W)$ be an ice quiver with potential. We call $W$, and also $(Q,F,W)$, \emph{irredundant} if each term of $W$ includes at least one unfrozen arrow.
\end{defn}

\begin{prop}
\label{p:irredundant}
Let $(Q,F,W)$ be an ice quiver with potential. Then there is an irredundant potential $W^\circ$ such that $\frjac{Q}{F}{W}\cong\frjac{Q}{F}{W^\circ}$.
\end{prop}
\begin{proof}
Collecting terms containing only frozen arrows, there is a unique expression $W=W^\circ+W^\partial$ in which $W^\circ$ is irredundant and $W^\partial$ is a potential on $F$. Then $\der{\alpha}{W^\partial}=0$ for any $\alpha\in Q_1\setminus F_1$, so $\frjac{Q}{F}{W}\cong\frjac{Q}{F}{W^\circ}$.
\end{proof}

The main reduction operation of this section is motivated by the fact that the ideal of $\cpa{\KK}{Q}$ generated by the cyclic derivatives $\partial_\alpha W$ may not be admissible, in the following sense.

\begin{defn}
Let $Q$ be a quiver. An ideal of $\cpa{\KK}{Q}$ is called \emph{admissible} if it is contained in $J^2=\prod_{d=2}^\infty(\KK Q_1)^{\tens_Sd}$. We call an ice quiver with potential $(Q,F,W)$ \emph{reduced} if $W$ is irredundant and the Jacobian ideal of $\cpa{\KK}{Q}$ determined by $F$ and $W$ is admissible.
\end{defn}

\begin{rem}
\label{r:gabriel}
Since we sometimes wish to consider infinite-dimensional algebras, our definition of admissibility differs from the usual one (e.g.\ \cite[Defn.~II.2.1]{assemelements}) by dropping the requirement that the ideal is contained in $J^n=\prod_{d=n}^\infty(\KK Q_1)^{\tens_Sd}$ for some $n\in\NN$. Since we defined $\frjac{Q}{F}{W}$ using the complete path algebra of $Q$, its quotient by its Jacobson radical is semisimple, and it has the idempotent lifting property (cf.\ Example~\ref{e:cpas}) and so it has a well-defined Gabriel quiver. It is a direct consequence of the definition that if $(Q,F,W)$ is reduced, this Gabriel quiver is $Q$.
\end{rem}

\begin{rem}
\label{r:dwz-comparison}
The Jacobian ideal of an irredundant potential $W$ is admissible if and only if no term of $W$ is a $2$-cycle (recalling that we already insist in the definition that no term of $W$ may be a loop). Thus our definition of reduced is equivalent to that of Amiot--Reiten--Todorov \cite[\S1.3]{amiotubiquity}, and also agrees with Derksen--Weyman--Zelevinsky's definition \cite[\S4]{derksenquivers1} in the case $F=\varnothing$, in which case every potential is automatically irredundant.
\end{rem}

The main result of this section allows us to replace any ice quiver with potential by a reduced one, without affecting the isomorphism class of the Jacobian algebra. This is a version of Derksen--Weyman--Zelevinsky's splitting theorem \cite[Thm.~4.6]{derksenquivers1} for ordinary quivers with potential. Indeed, our proof will be very similar, so we refer heavily to \cite{derksenquivers1} when the arguments apply essentially without change, focussing instead on where some adaptation is necessary to deal with frozen arrows.

\begin{thm}
\label{thm:reduction}
Let $(Q,F,W)$ be an ice quiver with potential. Then there exists a reduced ice quiver with potential $(Q_\red,F_\red,W_\red)$ such that $\frjac{Q}{F}{W}\cong\frjac{Q_\red}{F_\red}{W_\red}$.
\end{thm}

The proof of this theorem follows closely that of \cite[Thm.~4.6]{derksenquivers1}, and so we reproduce the necessary definitions and results from loc.\ cit., generalising to ice quivers with potential where necessary.

\begin{defn}[cf.~{\cite[Def.~4.2]{derksenquivers1}}]
\label{d:right-equiv}
Let $(Q,F)$ and $(Q',F')$ be ice quivers such that $Q_0=Q_0'$ and $F_0=F_0'$. In particular, this means that $\cpa{\KK}{Q}$ and $\cpa{\KK}{Q'}$ are complete tensor algebras over the same semisimple algebra $S=\KK Q_0$. An isomorphism $\varphi\colon\cpa{\KK}{Q}\to\cpa{\KK}{Q'}$ is said to be a \emph{right equivalence} of the ice quivers with potential $(Q,F,W)$ and $(Q',F',W')$ if
\begin{enumerate}[label=(\roman*)]
\item\label{rightequiv1} $\varphi|_S=\id{S}$,
\item\label{rightequiv2} $\varphi(\cpa{\KK}{F})=\cpa{\KK}{F'}$, where $\cpa{\KK}{F}$ and $\cpa{\KK}{F'}$ are treated in the natural way as subalgebras of $\cpa{\KK}{Q}$ and $\cpa{\KK}{Q'}$ respectively, and
\item\label{rightequiv3} $\varphi(W)$ is cyclically equivalent to $W'$.
\end{enumerate}
\end{defn}

\begin{rem}
\label{rem:inv-rightequiv}
If $\varphi$ is a right equivalence, then so is $\varphi^{-1}$. The right equivalences of $(Q,F,W)$ and $(Q',F',W')$ are precisely the right equivalences of the ordinary quivers with potential $(Q,W)$ and $(Q',W')$ \cite[Def.~4.2]{derksenquivers1} that also satisfy \ref{rightequiv2}, i.e.\ they take $\cpa{\KK}{F}$ to $\cpa{\KK}{F'}$. Thus the main way in which arguments from \cite{derksenquivers1} must be modified to fit our context is by ensuring that the necessary right equivalences can be chosen to respect the frozen subquivers.
\end{rem}

We introduce some more notation. Let
\[\cpa{\KK}{Q}\comptens_\KK\cpa{\KK}{Q}=\prod_{m,n\geq0}(\KK Q_1)^{\tens_Sm}\tens_\KK(\KK Q_1)^{\tens_Sn},\]
recalling that $S=\KK Q_0$. For any path $p=\alpha_k\dotsm\alpha_1$ of $Q$, and any $\alpha\in Q_1$, we may define
\[\Delta_\alpha(p)=\sum_{\alpha_i=\alpha}\alpha_k\dotsm\alpha_{i+1}\tens\alpha_{i-1}\dotsm\alpha_1\in\cpa{\KK}{Q}\comptens_\KK\cpa{\KK}{Q}\]
and extend by linearity and continuity to a map $\Delta_\alpha\colon\cpa{\KK}{Q}\to\cpa{\KK}{Q}\comptens_\KK\cpa{\KK}{Q}$. For $f\in\cpa{\KK}{Q}\comptens_\KK\cpa{\KK}{Q}$ and $g\in\cpa{\KK}{Q}$, we define $f\bullet g\in\cpa{\KK}{Q}$ by setting $(u\tens v)\bullet g=vgu$ and extending linearly. These definitions allow us to state a chain rule for cyclic derivatives, proved by Derksen--Weyman--Zelevinsky.

\begin{lem}[{\cite[Lem.~3.9]{derksenquivers1}}]
\label{chainrule}
If $Q$ and $Q'$ share a vertex set $Q_0$, and $\varphi\colon\cpa{\KK}{Q}\to\cpa{\KK}{Q'}$ is an algebra homomorphism restricting to the identity on $S=\KK Q_0$, then for any potential $W$ on $Q$ and any $\alpha\in Q_1'$, we have
\[\der{\alpha}{\varphi(W)}=\sum_{\beta\in Q_1}\Delta_\alpha(\varphi(\beta))\bullet\varphi(\der{\beta}{W}).\]
\end{lem}

\begin{prop}[cf.~{\cite[Prop.~3.7]{derksenquivers1}}]
If $\varphi$ is a right equivalence of $(Q,F,W)$ and $(Q',F',W')$, then $\varphi$ induces an isomorphism $\frjac{Q}{F}{W}\isoto\frjac{Q'}{F'}{W'}$.
\end{prop}
\begin{proof}
By \Cref{chainrule}, for any unfrozen arrow $\alpha$ of $Q'$, we have
\[\der{\alpha}{\varphi(W)}=\sum_{\beta\in Q_1}\Delta_\alpha(\varphi(\beta))\bullet\varphi(\der{\beta}{W}).\]
Since $\varphi$ restricts to an isomorphism $\cpa{\KK}{F}\isoto\cpa{\KK}{F'}$, if $\beta\in F_1$ then no term of $\varphi(\beta)$ can include the unfrozen arrow $\alpha$, and so we have $\Delta_\alpha(\varphi(\beta))=0$. Thus we may instead write
\[\der{\alpha}{\varphi(W)}=\sum_{\beta\in Q_1^m}\Delta_\alpha(\varphi(\beta))\bullet\varphi(\der{\beta}{W}),\]
and see that
\[\close{\Span{\der{\alpha}{W'}:\alpha\in Q'_1\setminus F_1'}}=\close{\Span{\der{\alpha}{\varphi(W)}:\alpha\in Q_1'\setminus F_1'}}\subseteq\close{\Span{\varphi(\der{\beta}{W}):\beta\in Q_1\setminus F_1}},\]
with the equality coming from the cyclic equivalence of $W'$ and $\varphi(W)$. Applying the same argument to $\varphi^{-1}$, which is also a right equivalence (Remark~\ref{rem:inv-rightequiv}), we obtain the reverse inclusion, and the result follows.
\end{proof}

If $Q$ and $Q'$ are quivers sharing the same vertex set $Q_0$, we can define $Q\dsum Q'$ to be the quiver with vertex set $Q_0$ and arrows $Q_1\sqcup Q'_1$. If $F\subseteq Q$ and $F'\subseteq Q'$ are subquivers, then we write
\[(Q,F)\dsum(Q',F')=(Q\dsum Q',F\cup F'),\]
where $F\union F'$ is the subquiver with vertex set $F_0\union F'_0$ and arrow set $F_1\union F'_1$; note that while the second union is necessarily disjoint, because of the definition of $(Q\oplus Q')_1$, the first may not be. Finally, if $W$ and $W'$ are potentials on $Q$ and $Q'$ respectively, we can define
\[(Q,F,W)\dsum(Q',F',W')=(Q\dsum Q',F\union F',W+W').\]

\begin{defn}[cf.\ {\cite[Def.~4.3]{derksenquivers1}}]
An ice quiver with potential $(Q,F,W)$ is \emph{trivial} if $\frjac{Q}{F}{W}=\KK Q_0$.
\end{defn}

\begin{rem}
Just as in \cite[Prop.~4.4]{derksenquivers1}, trivial ice quivers with potential are, up to right equivalence, those in which $Q_1$ has exactly $2N$ arrows $\alpha_1,\beta_1,\dotsc,\alpha_N,\beta_N$, all unfrozen, such that $\alpha_i\beta_i$ is a $2$-cycle for all $i$, and $W=\sum_{i=1}^N\alpha_i\beta_i$.

Note that if we allowed the square of a loop to be a term of $W$, this statement would be false; for $Q$ consisting of a single vertex and a loop $\alpha$, taking $W=\alpha^2$ gives $\jac{Q}{W}=\powser{\KK}{\alpha}/\close{\Span{2\alpha}}\cong\KK$ provided $\Char{\KK}\ne2$.
\end{rem}

\begin{prop}
\label{p:dsum-triv}
Let $(Q,F,W)$ and $(Q',F',W')$ be ice quivers with potential such that $Q_0=Q'_0$ and $F_0=F'_0$. If $(Q',F',W')$ is trivial, then the canonical map $\cpa{\KK}{Q}\to\cpa{\KK}{Q\dsum Q'}$ induces an isomorphism
\[\frjac{Q}{F}{W}\isoto\frjac{Q\dsum Q'}{F\union F'}{W+W'}.\]
\end{prop}

\begin{proof}
The proof is exactly as in the case that $F$ is empty \cite[Prop.~4.5]{derksenquivers1}, noting for ease of comparison that the triviality of $(Q',F',W')$ implies that $F'_1=\varnothing$.
\end{proof}

Before proving Theorem~\ref{thm:reduction}, we give one more lemma, which provides a normal form for irredundant potentials, up to right equivalence.

\begin{lem}
\label{l:neat-potential}
Let $(Q,F,W)$ be an ice quiver with potential such that $W$ is irredundant. Then, up to replacing $W$ by a right equivalent potential, we have
\begin{equation}
\label{eq:neat-potential}
W=\sum_{i=1}^M\alpha_i\beta_i+\sum_{i=M+1}^N(\alpha_i\beta_i+\alpha_ip_i)
+W_1
\end{equation}
for some arrows $\alpha_i$ and $\beta_i$ and elements $p_i\in J^2$, where
\begin{enumerate}[label=(\roman*)]
\item $\alpha_i$ is unfrozen for all $1\leq i\leq N$, and $\beta_i$ is frozen if and only if $i>M$,
\item the arrows $\alpha_i$ and $\beta_i$ with $1\leq i\leq M$ each appear exactly once in the expression \eqref{eq:neat-potential},
\item the arrows $\beta_i$, for $1\leq i\leq N$, do not appear in any of the $p_j$, and
\item the arrows $\alpha_i$ and $\beta_i$, for $1\leq i\leq N$, do not appear in the potential $W_1$, and this potential has no degree $2$ terms.
\end{enumerate}
\end{lem}

\begin{proof}
Up to cyclic equivalence and rescaling arrows, we have
\begin{equation}
\label{eq:gen-potential}
W=\sum_{i=1}^N(\alpha_i\beta_i+\alpha_ip_i+q_i\beta_i)+W_0
\end{equation}
for some $p_i,q_i\in J^2=\prod_{d=2}^\infty(\KK Q_1)^{\tens_Sd}$, such that the terms $\alpha_i\beta_i$ are the only $2$-cycles in $W$, and no term of $W_0$ contains $\alpha_i$ or $\beta_i$; cf.\ \cite[Eq.~4.6]{derksenquivers1}. It follows from our definition of an ice quiver with potential that none of the arrows $\alpha_i$ and $\beta_i$ appearing in $2$-cycles in $W$ are loops. We can label these arrows so that $\alpha_i$ is always unfrozen, and $\beta_i$ is unfrozen if and only if $1\leq i\leq M$ for some $M\leq N$, as required by (i). We can also arrange that the arrows $\beta_i$ do not appear in any of the $p_j$, as required by (iii)---to do this, we cyclically rotate any term of $\alpha_j p_j$ containing $\beta_i$ until it ends with this arrow, and relabel so that the rest of this term is incorporated into $q_i$ instead.

The proof of \cite[Lem.~4.7]{derksenquivers1} applies in our situation to show that, up to right equivalence, we may assume that $p_i=0$ whenever $\beta_i$ is unfrozen, and $q_i=0$ for all $i$. Indeed, this lemma can be used to construct a right equivalence $\varphi\colon\cpa{\KK}{Q}\isoto\cpa{\KK}{Q}$ which is the identity on vertices, all frozen arrows and all unfrozen arrows different from the $\alpha_i$ and $\beta_i$, and takes $W$ to a potential of the required form. The reader is warned that the various potential terms in \cite[Lem.~4.8]{derksenquivers1} should be relabelled as above to conform to our insistence that the $\beta_i$ do not appear in any of the $p_j$; this relabelling then affects the next term of the inductive sequence of right equivalences constructed in the proof of \cite[Lem.~4.7]{derksenquivers1}, whose limit is our desired equivalence $\varphi$.

After applying this equivalence, our potential has the form \eqref{eq:neat-potential}, and this expression satisfies conditions (i), (iii) and (iv). To also impose condition (ii), we apply further right equivalences as follows. First collect terms involving the arrow $\beta_1$, and write them as
\[\alpha_1\beta_1+\gamma\beta_1\]
for some linear combination of paths $\gamma$; indeed, because of our assumptions on the expression \eqref{eq:neat-potential}, $\gamma$ is even a sum of some of the unfrozen arrows $\alpha_i$ for $1\leq i\leq m$. Since $\alpha_1$ is unfrozen, there is a right equivalence $\varphi$ fixing all arrows except $\alpha_1$ and with $\varphi(\alpha_1)=\alpha_1-\gamma$. Since $\beta_1$ is not a loop, it cannot appear in $\gamma$. As a result, the right equivalent potential $\varphi(W)$ contains $\beta_1$ exactly once, and so after relabelling it still has an expression \eqref{eq:neat-potential} satisfying (i), (iii) and (iv). Note that since $\varphi(\alpha_1)$ does not involve any loops, nor can any degree $2$ term of $\varphi(W)$, and so this is a valid potential. Now we can collect the terms of $\varphi(W)$ involving $\alpha_1$, writing them as
\[\alpha_1\beta_1+\alpha_1\delta\]
for some linear combination $\delta$ of paths. Since $\beta_1$ is unfrozen, there is a right equivalence $\psi$ fixing all arrows different from $\beta_1$, and with $\psi(\beta_1)=\beta_1-\delta$. Since we already arranged that $\beta_1$ appears exactly once in the potential $\varphi(W)$, the right equivalent potential $\psi\varphi(W)$ is obtained by simply removing the terms $\alpha_1\delta$, so that in $\psi\varphi(W)$ the arrow $\alpha_1$ also appears exactly once, and this potential still has an expression of the form \eqref{eq:neat-potential} satisfying (i), (iii) and (iv). Now we can apply the same argument to the potential $\psi\varphi(W)-\alpha_1\beta_1$, which involves a strictly smaller number of the finitely many arrows of $Q$, and inductively obtain a potential of the required form.
\end{proof}

We are now ready to prove the main result of the section.

\begin{proof}[Proof of Theorem~\ref{thm:reduction}]
By Proposition~\ref{p:irredundant}, we may assume that $W$ is irredundant, and thus further assume that $W$ has an expression of the form \eqref{eq:neat-potential} satisfying conditions (i)--(iv) from Lemma~\ref{l:neat-potential}.

Take $Q_\triv$ to be the subquiver of $Q$ consisting of all vertices and the arrows $\alpha_i,\beta_i$ for $i\leq M$, and $W_\triv=\sum_{i=1}^M\alpha_i\beta_i$, so that $(Q_\triv,W_\triv)$ is trivial. Let $Q'$ be the subquiver of $Q$ consisting of all vertices and those arrows not included in $Q_\triv$, and $W'=W-W_\triv$; we have arranged things so that $W'$ does not involve any arrows of $Q_\triv$, and thus defines a potential on $Q'$. Then $(Q,F,W)=(Q',F,W')\dsum(Q_\triv,\varnothing,W_\triv)$, and hence $\frjac{Q}{F}{W}\cong\frjac{Q'}{F}{W'}$ by Proposition~\ref{p:dsum-triv}.

Thus to finish the proof, it is enough to find a reduced ice quiver with potential $(Q_\red,F_\red,W_\red)$ such that $\frjac{Q'}{F}{W'}\cong\frjac{Q_\red}{F_\red}{W_\red}$. Simplifying our expression for $W'$, and relabelling arrows for simplicity, we have
\[W'=\sum_{i=1}^K\alpha_i\beta_i+W_\red,\]
where each $\alpha_i$ is unfrozen, and each $\beta_i$ is frozen and does not appear in any term of $W_\red$---note that we used condition (iii) from Lemma~\ref{l:neat-potential} here.

Let $(Q_\red,F_\red)$ be the ice quiver obtained from $(Q',F)$ by deleting $\beta_i$ and freezing $\alpha_i$ for each $1\leq i\leq K$. Then, by construction, $(Q_\red,F_\red,W_\red)$ is reduced. We claim that the map $\varphi\colon\frjac{Q'}{F}{W'}\to\frjac{Q_\red}{F_\red}{W_\red}$ acting as the identity on vertices, and on arrows by
\[\varphi(\gamma)=\begin{cases}\gamma,&\text{$\gamma\ne\beta_i$ for any $1\leq i\leq K$},\\-\der{\alpha_i}{W_{\red}},&\gamma=\beta_i,\end{cases}\]
is an isomorphism.

First we check that $\varphi$ is well-defined. If $\gamma$ is unfrozen and not equal to $\alpha_i$ for any $i$, then
\[\varphi(\der{\gamma}{W'})=\varphi(\der{\gamma}{W_\red})=\der{\gamma}{W_\red}=0,\]
since $\beta_i$ does not appear in $W_\red$, and $\gamma$ is unfrozen in $Q'$. On the other hand,
\[\varphi(\der{\alpha_i}{W'})=\varphi(\beta_i+\der{\alpha_i}{W_\red})=-\der{\alpha_i}{W_\red}+\der{\alpha_i}{W_\red}=0.\]
To see that $\varphi$ is an isomorphism, let $\psi\colon\frjac{Q_\red}{F_\red}{W_\red}\to\frjac{Q'}{F}{W'}$ be the map acting as the identity on vertices and arrows. This is also well-defined, as for each unfrozen $\gamma$ in $Q_\red$ we have
\[\psi(\der{\gamma}{W_\red})=\der{\gamma}{W_\red}=\der{\gamma}{W'}=0,\]
as $\gamma$ is not one of the $\alpha_i$, which are arrows of $F_\red$. Moreover,
\[\psi(-\der{\alpha_i}{W_\red})=-\der{\alpha_i}{W_\red}=-\der{\alpha_i}{W'}+\beta_i=\beta_i\]
in $\frjac{Q'}{F}{W'}$, so $\psi$ and $\varphi$ are inverses.
\end{proof}

\begin{prop}
\label{p:reduction-uniqueness}
Let $(Q,F,W)$ be an irredundant ice quiver with potential. Then the ice quiver with potential $(Q_\red,F_\red,W_\red)$ from \Cref{thm:reduction} is uniquely determined up to right equivalence by the right equivalence class of $(Q,F,W)$.
\end{prop}
\begin{proof}
As in \cite[Prop.~4.9]{derksenquivers1}, if $(Q',F',W')$ and $(Q'',F'',W'')$ are reduced and $(Q_\triv,\varnothing,W_\triv)$ is a trivial ice quiver with potential such that $(Q'\oplus Q_\triv,F',W'+W_\triv)$ is right equivalent to $(Q''\oplus Q_\triv,F'',W''+W_\triv)$, then $(Q',F',W')$ is right equivalent to $(Q'',F'',W'')$. Indeed, the proof of this proposition goes through without change in our more general setting, with the key lemmas in fact now being more general than we need---for example, \cite[Lem.~4.11]{derksenquivers1} is only used in the case that $b_k=0$ whenever $a_k$ is a frozen arrow.

Now let $(Q,F,W)$ be an irredundant ice quiver with potential. Then, as in \cite{derksenquivers1}, it is clear from the construction that the trivial quiver with potential $(Q_\triv,\varnothing,W_\triv)$ from the proof of \Cref{thm:reduction} is determined up to right equivalence by the right equivalence class of $(Q,F,W)$, and then the statement of the previous paragraph implies that the same is true of $(Q_\red,F_\red,W_\red)$.
\end{proof}

\begin{defn}
When $(Q,F,W)$ is an irredundant ice quiver with potential, bearing in mind \Cref{p:reduction-uniqueness}, we call $(Q_\red,F_\red,W_\red)$ from \Cref{thm:reduction} the \emph{reduction} of $(Q,F,W)$.
\end{defn}

\begin{eg}
\label{eg:reduction}
We give a simple example of reduction, illustrating the additional feature appearing in the case of ice quivers. Consider the ice quiver
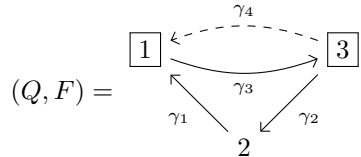
\[(Q,F)=\begin{tikzpicture}[baseline=0,scale=1.3]
\node at (0,0.5) (1) {$\boxed{1}$};
\node at (2,0.5) (3) {$\boxed{3}$};
\node at (1,-0.5) (2) {$2$};
\path[-angle 90,font=\scriptsize]
	(2) edge node[below left] {$\gamma_1$} (1)
	(3) edge node[below right] {$\gamma_2$} (2)
	(1) edge[bend right=20] node[below] {$\gamma_3$} (3);
\path[\frozen,-angle 90,font=\scriptsize]
	(3) edge[bend right=20] node[above] {$\gamma_4$} (1);
\end{tikzpicture}\]
in which $F$ consists of the boxed vertices $1$ and $3$ and the arrow $\gamma_3$, with potential $W=\gamma_1\gamma_2\gamma_3+\gamma_3\gamma_4$. This ice quiver with potential is not reduced; while $W$ is irredundant, the relation
\[\der{\gamma_3}{W}=\gamma_1\gamma_2+\gamma_4\]
includes a term consisting of a single arrow, and hence is not admissible. Following the proof of Theorem~\ref{thm:reduction}, we rewrite
\[W=\gamma_3\gamma_4+\gamma_3\gamma_1\gamma_2.\]
This expression satisfies the conditions of Lemma~\ref{l:neat-potential}; in the notation of \eqref{eq:neat-potential}, $N=1$, and we have $\alpha_1=\gamma_3$, $\beta_1=\gamma_4$, $p_1=\gamma_1\gamma_2$, and $q_1=W_0=0$. Since $\beta_1=\gamma_4$ is frozen, we do not require that $p_1=0$.

Since there are no $2$-cycles in $W$ consisting only of unfrozen arrows, the quiver $(Q',F,W')$ constructed in the proof of Theorem~\ref{thm:reduction} is just $(Q,F,W)$ as above; that is, there is no trivial part to split off. The proof then tells us that $(Q,F,W)$ has frozen Jacobian algebra isomorphic to that of the reduced ice quiver with potential $(Q_\red,F_\red,W_\red)$ where $W_\red=\gamma_3\gamma_1\gamma_2$ (the term of $W$ not given by a $2$-cycle), and $(Q_\red,F_\red)$ is obtained by deleting $\gamma_4$, the frozen arrow appearing in the $2$-cycle in $W$, and freezing $\gamma_3$, the unfrozen arrow in this term. That is,
\[(Q_\red,F_\red)=\mathord{\begin{tikzpicture}[baseline=0,scale=1.3]
\node at (0,0.5) (1) {$\boxed{1}$};
\node at (2,0.5) (3) {$\boxed{3}$};
\node at (1,-0.5) (2) {$2$};
\path[-angle 90,font=\scriptsize]
	(2) edge node[below left] {$\gamma_1$} (1)
	(3) edge node[below right] {$\gamma_2$} (2);
\path[\frozen,-angle 90,font=\scriptsize]
	(1) edge node[above] {$\gamma_3$} (3);
\end{tikzpicture}}\]
The reader may readily check that there is an isomorphism $\frjac{Q}{F}{W}\isoto\frjac{Q_\red}{F_\red}{W_\red}$, given by the recipe in the proof of Theorem~\ref{thm:reduction}.
\end{eg}

\section{Mutation}
\label{s:mutation}

In this section, we discuss how ice quivers with potential transform under a local move at a mutable vertex, called a \emph{mutation}. Unlike the operations in Section~\ref{s:reduction}, this does not leave the isomorphism class of the Jacobian algebra invariant in general, even in the case that $F=\varnothing$. The name is chosen because of the connection to mutation in cluster algebras, which we discuss in more detail below, and corresponds to the physical operation of Seiberg duality \cites{berensteinseiberg,mukhopadhyayseiberg} (sometimes called urban renewal) on the dimer models discussed in Example~\ref{e:dimers1}; see for example \cite{vitoriamutations}. As in Section~\ref{s:reduction}, many of the arguments in this section carry over essentially without change from those in \cite{derksenquivers1} for the unfrozen case, so we focus on pointing out where modifications are necessary.

\begin{defn}
\label{iqpmutation}
Let $(Q,F,W)$ be an irredundant ice quiver with potential, and let $v\in Q_0\setminus F_0$ be a mutable vertex such that no loops or $2$-cycles of $Q$ are incident with $v$. Then the ice quiver with potential $\mu_v(Q,F,W)=(\mu_vQ,\mu_vF,\mu_vW)$, called the \emph{mutation} of $(Q,F,W)$ at $v$, is the output of the following procedure.
\begin{enumerate}[label=(\roman*)]
\item\label{QPm1} For each pair of arrows $\alpha\colon u\to v$ and $\beta\colon v\to w$, add an unfrozen `composite' arrow $[\beta\alpha]\colon u\to w$ to $Q$. Since $v$ is not incident with loops or $2$-cycles, $[\beta\alpha]$ cannot be a loop.
\item\label{QPm2} Replace each arrow $\alpha\colon u\to v$ by an arrow $\alpha^*\colon v\to u$, and each arrow $\beta\colon v\to w$ by an arrow $\beta^*\colon w\to v$; these arrows are necessarily unfrozen, since $v$ is.
\item\label{QPm3} Pick a representative of $W$ in $\cpa{\KK}{Q}$ such that no term of $W$ begins at $v$ (which is possible since there are no loops at $v$). For each pair of arrows $\alpha,\beta$ as in \ref{QPm1}, replace each occurrence of $\beta\alpha$ in $W$ by $[\beta\alpha]$, and add the term $[\beta\alpha]\alpha^\vee\beta^\vee$---each term of the resulting potential still has degree at least $2$, since $\beta\alpha$ cannot be a $2$-cycle, and no degree $2$ terms involve loops, since $[\beta\alpha]$ is not a loop. This potential is also irredundant, since the arrows $[\beta\alpha]$ are unfrozen, but it need not be reduced even if $(Q,F,W)$ is.
\item\label{QPm4} Replace the resulting ice quiver with potential by its reduction, as in Theorem~\ref{thm:reduction}, this being unique up to right equivalence by \Cref{p:reduction-uniqueness}.
\end{enumerate}
It will sometimes be useful to consider the ice quiver with potential $\lift{\mu}_v(Q,F,W)=(\lift{\mu}_vQ,F,\lift{\mu}_vW)$ obtained after steps \ref{QPm1}--\ref{QPm3}, i.e.\ before taking the reduction. Note that $\lift{\mu}_v(Q,F,W)$ and $\mu_v(Q,F,W)$ define isomorphic frozen Jacobian algebras, by Theorem~\ref{thm:reduction}. When $F=\varnothing$, this definition of mutation agrees with that given by Derksen--Weyman--Zelevinsky \cite[\S5]{derksenquivers1}, by \Cref{r:dwz-comparison}. The quiver $\mu_vQ$ may have $2$-cycles, even if $Q$ did not, although these cannot be incident with $v$.
\end{defn}

For any ice quiver with potential $(Q,F,W)$, let $(\stab{Q},\stab{W})$ be the quiver with potential in which $\stab{Q}$ is the full subquiver of $Q$ on the mutable vertices, and $\stab{W}$ is the image of $W$ under the canonical quotient map $\cpa{\KK}{Q}/\close{\{\cpa{\KK}{Q},\cpa{\KK}{Q}\}}\to\cpa{\KK}{\stab{Q}}/\close{\{\cpa{\KK}{\stab{Q}},\cpa{\KK}{\stab{Q}}\}}$. If $v$ is a mutable vertex of $Q$, and hence also a vertex of $\stab{Q}$, it is immediate from the definition that $(\stab{\lift{\mu}_vQ},\stab{\lift{\mu}_vQ})=\lift{\mu}_v(\stab{Q},\stab{W})$. Bearing in mind \Cref{r:dwz-comparison}, we even have $(\stab{\mu_vQ},\stab{\mu_vW})=\mu_v(\stab{Q},\stab{W})$.

We now study further properties of the operations $\lift{\mu}_v$ and $\mu_v$. Since $\lift{\mu}_v$ does not affect the frozen subquiver (and agrees with $\mu_v$ after taking frozen Jacobian algebras), we will be able to import even more arguments directly from \cite{derksenquivers1} than in \Cref{s:reduction}. For example, we have the following.

\begin{prop}[cf.~{\cite[Cor.~5.4]{derksenquivers1}}]
Let $(Q,F,W)$ be an ice quiver with potential and $v\in Q_0\setminus F_0$ a mutable vertex not incident with any loops or $2$-cycles of $Q$. Then the right equivalence class of $\mu_v(Q,F,W)$ is determined by that of $(Q,F,W)$.
\end{prop}
\begin{proof}
The proof of \cite[Thm.~5.2]{derksenquivers1}, for the case that $F=\varnothing$, is local to the mutable vertex $v$, and so applies equally well in this case to show that the right equivalence class of $\lift{\mu}_v(Q,F,W)$ is determined by that of $(Q,F,W)$; in particular, the right equivalences of $(\lift{\mu}_vQ,\lift{\mu}_vW)$ constructed in this argument are also right equivalences of $(\lift{\mu}_vQ,\lift{\mu}_vF,\lift{\mu}_vW)$, since there are no arrows of $F$ incident with $v$. The result then follows from \Cref{p:reduction-uniqueness}.
\end{proof}

\begin{thm}[cf.~{\cite[Thm.~5.7]{derksenquivers1}}]
\label{thm:mu-involution}
Let $(Q,F,W)$ be a reduced ice quiver with potential and $v\in Q_0\setminus F_0$ a mutable vertex not incident with any $2$-cycles. Then $\mu_v^2(Q,F,W)$ is right-equivalent to $(Q,F,W)$.
\end{thm}

\begin{proof}
The proof of \cite[Thm.~5.7]{derksenquivers1} also applies here, as follows. The quiver $\lift{\mu}_v^2Q$ differs from $Q$ only by the addition of a $2$-cycle, consisting of unfrozen arrows $[\beta\alpha]\colon u\to w$ and $[\alpha^*\beta^*]\colon w\to u$, for each pair of arrows $\alpha\colon u\to v$ and $\beta\colon v\to w$ of $Q$. (Formally, we also make the identifications $(\alpha^*)^*=\alpha$ and $(\beta^*)^*=\beta$ for these arrows.) Moreover, the frozen subquiver of $\lift{\mu}^2_v(Q,F,W)$ is $F$ by definition, and 
\[\lift{\mu}_v^2W=[W]+\sum_{\substack{\alpha\colon u\to v\\\beta\colon v\to w}}([\beta\alpha]+\beta\alpha)[\alpha^*\beta^*],\]
where $[W]$ is obtained from $W$ by replacing each occurrence of a path $\beta\alpha$ through $v$ by $[\beta\alpha]$, as in \Cref{iqpmutation}\ref{QPm3} (cf.~\cite[Eq.~5.20]{derksenquivers1}).

Now let $Q_\triv$ be the subquiver of $Q$ consisting of all vertices and the arrows $[\beta\alpha]$ and $[\alpha^*\beta^*]$, and
\[W_\triv=\sum_{\substack{\alpha\colon u\to v\\\beta\colon v\to w}}[\beta\alpha][\alpha^*\beta^*],\]
so that $(Q_\triv,\varnothing,W_\triv)$ is a trivial ice quiver with potential. To conclude the argument, one can show exactly as in \cite[Thm.~5.7]{derksenquivers1} that $\lift{\mu}_v^2(Q,F,W)$ is right equivalent to $(Q,F,W)\oplus(Q_\triv,\varnothing,W_\triv)$. Indeed, the three right equivalences of ordinary quivers with potential constructed by this argument act as the identity on all arrows of $\lift{\mu}_v^2Q$ except those incident with $v$ or of the form $[\beta\alpha]$ or $[\alpha^*\beta^*]$, all of which are unfrozen. Since $(Q,F,W)$ is reduced by assumption, it follows from \Cref{p:reduction-uniqueness} that $\mu_v^2(Q,F,W)=(Q,F,W)$, as required.
\end{proof}

We now compare mutation of ice quivers with potential to the combinatorial process of Fomin--Zelevinsky mutation (see \cite[Def.~4.2]{fomincluster1} for the original definition in terms of matrices, or \cite[\S3.2]{kellercluster} in the language of quivers). First we recall this procedure, while also extending it slightly to cover the situation of ice quivers $(Q,F)$ which have arrows between their frozen vertices.

\begin{defn}
\label{fzmutation}
Let $(Q,F)$ be an ice quiver, and let $v\in Q_0\setminus F_0$ be a mutable vertex not incident with any loops or $2$-cycles. Then the \emph{extended Fomin--Zelevinsky mutation} $\mu^\FZ_v(Q,F)=(\mu^\FZ_vQ,\mu^\FZ_vF)$ of $(Q,F)$ at $v$ is defined to be the output of the following procedure.
\begin{enumerate}[label=(\roman*)]
\item\label{FZm1} For each pair of arrows $\alpha\colon u\to v$ and $\beta\colon v\to w$, add an unfrozen arrow $[\beta\alpha]\colon u\to w$ to $Q$.
\item\label{FZm2} Replace each arrow $\alpha\colon u\to v$ by an arrow $\alpha^*\colon v\to u$, and each arrow $\beta\colon v\to w$ by an arrow $\beta^*\colon w\to v$.
\item\label{FZm3} Remove a maximal collection of unfrozen $2$-cycles, i.e.\ $2$-cycles avoiding the subquiver $F$.
\item\label{FZm4} Choose a maximal collection of half-frozen $2$-cycles, i.e.\ $2$-cycles in which precisely one arrow is frozen. Replace each $2$-cycle in this collection by a frozen arrow, in the direction of the unfrozen arrow in the $2$-cycle.
\end{enumerate}
Note that, because of the choices involved in steps \ref{FZm3} and \ref{FZm4}, this operation is only defined up to quiver isomorphism. If we ignore all arrows between frozen vertices, as is typical in cluster theory, then step \ref{FZm4} has no effect, and we obtain the usual (unextended) definition of Fomin--Zelevinsky mutation.
\end{defn}

\begin{eg}
The rules for removing $2$-cycles in steps \ref{FZm3} and \ref{FZm4} of \Cref{fzmutation} appear naturally as equivalences of dimer models on surfaces with boundary, where they correspond to integrating out massive terms---on the bipartite graph, this amounts to removing a bivalent vertex and, if this vertex is not incident with a half-edge, merging the two adjacent vertices. This has the effect on the quiver of removing the $2$-cycle around this vertex when both arrows in this cycle are unfrozen, and replacing it by a frozen arrow (with the predicted orientation) if one of the arrows is frozen; see \Cref{fig:massive-terms} (cf.\ \cite[Lem.~12.1]{baurdimer}).
\begin{figure}[h]
$\mathord{\begin{tikzpicture}[baseline={(current bounding box.center)}]
\node at (-1,0) (lf) {};
\node at (1,0) (rf) {};
\path
	(0,0.75) edge (0,0) edge (0,-1)
	(0,0.75) edge (-1,1.75)
	(0,0.75) edge (0,1.75)
	(0,0.75) edge (1,1.75)
	(0,-0.75) edge (-1,-1.75)
	(0,-0.75) edge (0,-1.75)
	(0,-0.75) edge (1,-1.75);
\draw[black,fill=black,radius=3pt] (0,0) circle;
\draw[black,fill=white,radius=3pt] (0,0.75) circle;
\draw[black,fill=white,radius=3pt] (0,-0.75) circle;
\path[-angle 90]
	(rf) edge [bend right] (lf)
	(lf) edge [bend right] (rf)
	(lf) edge (-0.25,1.5)
	(0.25,1.5) edge (rf)
	(rf) edge (0.25,-1.5)
	(-0.25,-1.5) edge (lf);
\path[dotted]
	(-0.25,1.5) edge (0.25,1.5)
	(-0.25,-1.5) edge (0.25,-1.5);
\end{tikzpicture}}\quad\mapsto\quad
\mathord{\begin{tikzpicture}[baseline={(current bounding box.center)}]
\node at (-1,0) (lf) {};
\node at (1,0) (rf) {};
\path
	(0,0) edge (-1,1)
	(0,0) edge (0,1)
	(0,0) edge (1,1)
	(0,0) edge (-1,-1)
	(0,0) edge (0,-1)
	(0,0) edge (1,-1);
\draw[black,fill=white,radius=3pt] (0,0) circle;
\path[-angle 90]
	(lf) edge (-0.25,0.75)
	(0.25,0.75) edge (rf)
	(rf) edge (0.25,-0.75)
	(-0.25,-0.75) edge (lf);
\path[dotted]
	(-0.25,0.75) edge (0.25,0.75)
	(-0.25,-0.75) edge (0.25,-0.75);
\end{tikzpicture}}\qquad\qquad
\mathord{\begin{tikzpicture}[baseline={(current bounding box.center)}]
\node at (-1,0) (lf) {};
\node at (1,0) (rf) {};
\path
	(0,0.75) edge (0,0) edge (0,-0.75)
	(0,0.75) edge (-1,1.75)
	(0,0.75) edge (0,1.75)
	(0,0.75) edge (1,1.75);
\draw[black,fill=black,radius=3pt] (0,0) circle;
\draw[black,fill=white,radius=3pt] (0,0.75) circle;
\path[-angle 90]
	(rf) edge [bend right] (lf)
	(lf) edge (-0.25,1.5)
	(0.25,1.5) edge (rf);
\path[-angle 90,dashed]
	(lf) edge [bend right] (rf);
\path[dotted]
	(-0.25,1.5) edge (0.25,1.5);
\path[ultra thick,\bdrycolor]
	(-1.25,-0.75) edge (1.25,-0.75);
\end{tikzpicture}}\quad\mapsto\quad
\mathord{\begin{tikzpicture}[baseline={(current bounding box.center)}]
\node at (-1,0) (lf) {};
\node at (1,0) (rf) {};
\path
	(0,0.75) edge (0,0) edge (0,-0.75)
	(0,0) edge (-1,1)
	(0,0) edge (0,1)
	(0,0) edge (1,1);
\draw[black,fill=white,radius=3pt] (0,0) circle;
\path[-angle 90]
	(lf) edge (-0.25,0.75)
	(0.25,0.75) edge (rf);
\path[-angle 90,dashed]
	(rf) edge [bend left] (lf);
\path[dotted]
	(-0.25,0.75) edge (0.25,0.75);
\path[ultra thick,\bdrycolor]
	(-1.25,-0.75) edge (1.25,-0.75);
\end{tikzpicture}}$
\caption{\label{fig:massive-terms} Removing bivalent vertices from a dimer model, either in the interior (left) or at the boundary (right), and the effect on the dual ice quiver. (The colours of the vertices are not important, and can be swapped, causing the orientations of arrows to be reversed.)}
\end{figure}
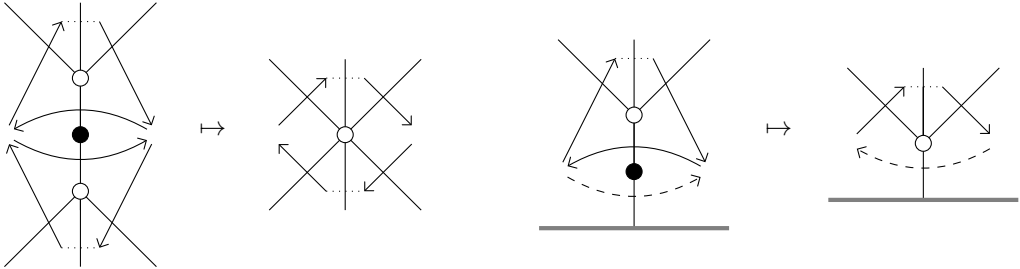

\end{eg}

Our goal now is to understand conditions on a potential $W$ such that the ice quiver $(\mu_vQ,\mu_vF)$ of $\mu_v(Q,F,W)$ coincides with the extended Fomin--Zelevinsky mutation $\mu_v^\FZ(Q,F)$. Since the first two steps of the two mutation procedures are the same, we need only decide when \Cref{iqpmutation}\ref{QPm4}, the reduction step, induces the required cancellation of $2$-cycles in \Cref{fzmutation}\ref{FZm3}--\ref{FZm4}. The first observation is essentially immediate from the proof of \Cref{thm:reduction}.

\begin{prop}
\label{p:no-2-cycles}
Let $(Q,F,W)$ be an ice quiver with potential, and $v\in Q_0\setminus F_0$. If $(\mu_vQ,\mu_vF)$ has no $2$-cycles containing unfrozen arrows, then it agrees with $\mu_v^\FZ(Q,F)$.
\end{prop}
\begin{proof}
The ice quiver of $\lift{\mu}_v(Q,F,W)$ is the result of applying the first two steps of the operation $\mu_v^\FZ$ to $(Q,F)$. The construction of the reduction of $\lift{\mu}_v(Q,F,W)$, given in the proof of \Cref{thm:reduction}, modifies the ice quiver only by removing unfrozen $2$-cycles, as in \Cref{fzmutation}\ref{FZm3}, and performing the replacement operation on half-frozen $2$-cycles described in \Cref{fzmutation}\ref{FZm4}. Thus if no $2$-cycles of these types remain in $(\mu_vQ,\mu_vF)$, the collections of such $2$-cycles that were removed or replaced as part of the reduction process must have been maximal, and so $(\mu_vQ,\mu_vF)=\mu_v^\FZ(Q,F)$.
\end{proof}

\begin{defn}
\label{d:non-degenerate}
We say that an ice quiver with potential $(Q,F,W)$ is \emph{non-degenerate} if, for any $(Q',F',W')$ obtained from $(Q,F,W)$ by a sequence of mutations, $(Q',F')$ has no $2$-cycles containing unfrozen arrows, or equivalently if $(Q',F')$ coincides precisely with the result of performing the corresponding sequence of extended Fomin--Zelevinsky mutations to $(Q,F)$.
\end{defn}

This condition is typically difficult to check---for a quiver with potential whose Jacobian algebra is the endomorphism algebra of a cluster-tilting object in a suitable category, sufficient conditions for non-degeneracy are given in Section~\ref{s:ct-objects}. Alternatively, non-degeneracy is implied by the following algebraic condition on $\frjac{Q}{F}{W}$.

\begin{defn}
\label{rigid}
The \emph{trace space} of $A=\frjac{Q}{F}{W}$ is
\[\Tr(A)=A/\close{\{A,A\}},\]
where $\{A,A\}$ denotes the vector subspace of $A$ spanned by commutators as in Definition~\ref{d:potential}. Abusing notation by denoting the image of $\cpa{\KK}{F}$ under the projection from $\cpa{\KK}{Q}$ to $\Tr(A)$ again by $\cpa{\KK}{F}$, and recalling that we may treat $S$ as the subalgebra of $A$ spanned by the vertex idempotents, the \emph{deformation space} of $W$ is $\Def(Q,F,W):=\Tr(A)/(S+\cpa{\KK}{F})$. We call $W$ a \emph{rigid} potential for $(Q,F)$ if $\Def(Q,F,W)=0$.
\end{defn}

\begin{rem}
\label{r:rigid-equiv}
A potential $W$ for $(Q,F)$ is rigid if and only if every cycle in $Q$ containing an unfrozen arrow is cyclically equivalent to an element of the Jacobian ideal; cf.\ \cite[(8.1)]{derksenquivers1}.
\end{rem}

\begin{prop}[{cf.~\cite[Prop.~8.1, Cor.~6.11]{derksenquivers1}}]
\label{rigidmutation}
If $(Q,F,W)$ is rigid and reduced, then it has no $2$-cycles containing unfrozen arrows. Moreover, all of its mutations are also rigid and reduced.
\end{prop}
\begin{proof}
The proof of \cite[Prop.~8.1]{derksenquivers1} applies in this context to show that $(Q,F,W)$ has no $2$-cycles containing unfrozen arrows; in short, such a cycle would violate the condition of \Cref{r:rigid-equiv} implied by rigidity of $W$. Since we only allow mutation at mutable vertices, and reduction does not affect the isomorphism class of the Jacobian algebra by \Cref{thm:reduction}, the proof of \cite[Cor.~6.11]{derksenquivers1} also applies to show that all mutations of $(Q,F,W)$ are rigid. They are reduced by definition.
\end{proof}

Combining this with \Cref{p:no-2-cycles}, we immediately obtain the following corollary.

\begin{cor}[cf.~{\cite[Prop.~7.1]{derksenquivers1}}]
\label{p:fzmutation}
Let $(Q,F,W)$ be rigid and reduced. Then for any mutable vertex $v$ of $Q$, the ice quiver $(\mu_vQ,\mu_vF)$ of $\mu_v(Q,F,W)$ agrees with the extended Fomin--Zelevinsky mutation $\mu^\FZ_v(Q,F)$.
\end{cor}

\begin{eg}
\label{QPmutation}
Consider the ice quiver with potential $(Q,F,W)$ given by
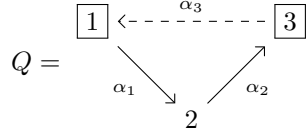
\[Q=\mathord{\begin{tikzpicture}[baseline=0,scale=1.3]
\node at (0,0.5) (1) {$\boxed{1}$};
\node at (2,0.5) (3) {$\boxed{3}$};
\node at (1,-0.5) (2) {$2$};
\path[-angle 90,font=\scriptsize]
	(1) edge node[below left] {$\alpha_1$} (2)
	(2) edge node[below right] {$\alpha_2$} (3);
\path[\frozen,-angle 90,font=\scriptsize]
	(3) edge node[above] {$\alpha_3$} (1);
\end{tikzpicture}}\]
where $F$ is the full subquiver on $\{1,3\}\subseteq Q_0$; we denote frozen subquivers in this way, with boxed vertices and dashed arrows, throughout the example. We pick the potential $W=\alpha_3\alpha_2\alpha_1$, which is reduced since every cycle in $Q$ is cyclically equivalent to an element of its Jacobian ideal $\close{\Span{\alpha_1\alpha_3,\alpha_3\alpha_2}}$. Mutating at vertex $2$ produces
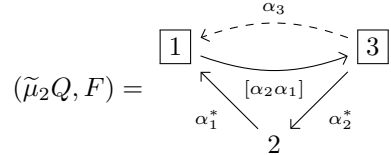
\[(\lift{\mu}_2Q,F)=\begin{tikzpicture}[baseline=0,scale=1.3]
\node at (0,0.5) (1) {$\boxed{1}$};
\node at (2,0.5) (3) {$\boxed{3}$};
\node at (1,-0.5) (2) {$2$};
\path[-angle 90,font=\scriptsize]
	(2) edge node[below left] {$\alpha_1^*$} (1)
	(3) edge node[below right] {$\alpha_2^*$} (2)
	(1) edge[bend right=20] node[below] {$[\alpha_2\alpha_1]$} (3);
\path[\frozen,-angle 90,font=\scriptsize]
	(3) edge[bend right=20] node[above] {$\alpha_3$} (1);
\end{tikzpicture}\]
with potential $\lift{\mu}_2W=\alpha_2^*[\alpha_2\alpha_1]\alpha_1^*+\alpha_3[\alpha_2\alpha_1]$; the only frozen arrow is $\alpha_3$. This ice quiver with potential is not reduced, but $\mu_2(Q,F,W)$ is given by its reduction, which is the ice quiver
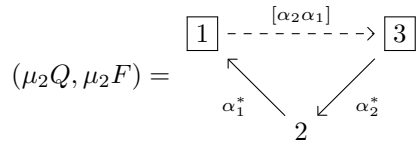
\[(\mu_2Q,\mu_2F)=\mathord{\begin{tikzpicture}[baseline=0,scale=1.3]
\node at (0,0.5) (1) {$\boxed{1}$};
\node at (2,0.5) (3) {$\boxed{3}$};
\node at (1,-0.5) (2) {$2$};
\path[-angle 90,font=\scriptsize]
	(2) edge node[below left] {$\alpha_1^*$} (1)
	(3) edge node[below right] {$\alpha_2^*$} (2);
\path[\frozen,-angle 90,font=\scriptsize]
	(1) edge node[above] {$[\alpha_2\alpha_1]$} (3);
\end{tikzpicture}}\]
with potential $\mu_2W=\alpha_2^*[\alpha_2\alpha_1]\alpha_1^*$, as computed in Example~\ref{eg:reduction}.
Note that this ice quiver is the extended Fomin--Zelevinsky mutation $\mu^\FZ_2(Q,F)$, as predicted by \Cref{p:fzmutation}. Mutating at $2$ again gives
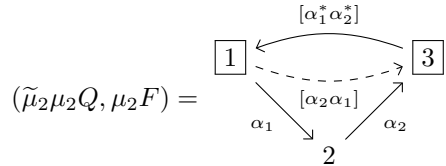
\[(\lift{\mu}_2\mu_2Q,\mu_2F)=\mathord{\begin{tikzpicture}[baseline=0,scale=1.3]
\node at (0,0.5) (1) {$\boxed{1}$};
\node at (2,0.5) (3) {$\boxed{3}$};
\node at (1,-0.5) (2) {$2$};
\path[-angle 90,font=\scriptsize]
	(1) edge node[below left] {$\alpha_1$} (2)
	(2) edge node[below right] {$\alpha_2$} (3)
	(3) edge[bend right=20] node[above] {$[\alpha_1^*\alpha_2^*]$} (1);
\path[\frozen,-angle 90,font=\scriptsize]
	(1) edge[bend right=20] node[below] {$[\alpha_2\alpha_1]$} (3);
\end{tikzpicture}}\]
where we write $(\alpha_i^*)^*=\alpha_i$ for simplicity. The potential is $\lift{\mu}_2\mu_2W=[\alpha_2\alpha_1][\alpha_1^*\alpha_2^*]+[\alpha_1^*\alpha_2^*]\alpha_2\alpha_1$, so a similar reduction gives
\[(\mu_2^2Q,\mu_2^2F)=\mathord{\begin{tikzpicture}[baseline=0,scale=1.3]
\node at (0,0.5) (1) {$\boxed{1}$};
\node at (2,0.5) (3) {$\boxed{3}$};
\node at (1,-0.5) (2) {$2$};
\path[-angle 90,font=\scriptsize]
	(1) edge node[below left] {$\alpha_1$} (2)
	(2) edge node[below right] {$\alpha_2$} (3);
\path[\frozen,-angle 90,font=\scriptsize]
	(3) edge node[above] {$[\alpha_1^*\alpha_2^*]$} (1);
\end{tikzpicture}}\]
with potential $[\alpha_1^*\alpha_2^*]\alpha_2\alpha_1$. Thus we have recovered the original ice quiver with potential, as predicted by \Cref{thm:mu-involution}.
\end{eg}

\section{Cluster-tilting objects}
\label{s:ct-objects}

In this section we discuss the mutation of cluster-tilting objects in Frobenius cluster categories, beginning with some definitions.

\begin{defn}
Let $\cat$ be an additive category, and $\subcat$ a full subcategory. A \emph{left $\subcat$-approximation} of an object $X\in\cat$ is a morphism $f\colon X\to D$ such that $D\in\subcat$ and $\Hom_\cat(f,D')$ is surjective for any $D'\in\subcat$. It is \emph{minimal} if any endomorphism $g\colon D\to D$ such that $gf=f$ is an isomorphism. A \emph{minimal right $\subcat$-approximation} is defined dually.
\end{defn}

\begin{defn}
An exact category $\frobcat$ \cite{buehlerexact} is a \emph{Frobenius category} if it has enough projective and injective objects, and these two classes of objects coincide.
\end{defn}

\begin{thm}[{\cite[\S I.2]{happeltriangulated}}]
Let $\frobcat$ be a Frobenius category. Then the stable category $\stab{\frobcat}$, formed by factoring out all morphisms factoring over a projective object, is a triangulated category, with suspension induced from the inverse syzygy functor $\Omega^{-1}$, taking the cokernel of an injective envelope.
\end{thm}

\begin{defn}
A triangulated category $\cat$ with suspension functor $\Sigma$ is said to be \emph{$d$-Calabi--Yau} if $\Sigma^d$ is a Serre functor, i.e.\ there are isomorphisms
\[\Hom_{\cat}(X,Y)\iso\Kdual\Hom_{\cat}(\Sigma^dY,X),\]
functorial in $X,Y\in\mathcal{C}$, where $D=\Hom_\KK(-,\KK)$. We call a Frobenius category $\frobcat$ \emph{stably $d$-Calabi--Yau} if the stable category $\stab{\frobcat}$ is $d$-Calabi--Yau.
\end{defn}

\begin{defn}
Let $\cat$ be a triangulated or exact category. We call $T\in\cat$ \emph{cluster-tilting} if
\[\{X\in\cat:\Ext^1_\cat(X,T)=0\}=\add{T}=\{Y\in\cat:\Ext^1_\cat(T,Y)=0\}.\]
\end{defn}

\begin{defn}[{\cite[Defn.~3.3]{presslandinternally}}]
\label{d:frob-clust-cat}
A \emph{Frobenius cluster category} is a Krull--Schmidt\footnote{This strengthens the original definition in \cite{presslandinternally}, which required only idempotent completeness.} stably $2$-Calabi--Yau Frobenius category $\frobcat$ with cluster-tilting objects, such that $\gldim\Endalg{\frobcat}{T}\leq 3$ for any cluster-tilting object $T\in\frobcat$.
\end{defn}

Most of the results in this section hold in more general situations than that of Definition~\ref{d:frob-clust-cat}, since we will not require the assumption on the global dimension of endomorphism algebras (although this assumption is not as strong as it might appear, and holds for most examples of Frobenius categorifications of cluster algebras; see Example~\ref{applications}).

Let $\frobcat$ be a Frobenius category, and let $T=\bigoplus_{k=1}^nT_k\in\frobcat$ be a cluster-tilting object. If $\stab{\frobcat}$ is $2$-Calabi--Yau, then Iyama--Yoshino's mutation theory for such triangulated categories \cite{iyamamutation} allows us to mutate $T\in\stab{\frobcat}$ at any indecomposable summand $T_k$---the indecomposable summands of $T$ in the stable category are precisely the non-projective indecomposable summands of $T\in\frobcat$. This induces a mutation of $T$ in the Frobenius category $\frobcat$, at any non-projective indecomposable summand $T_k$, summarised as follows.

\begin{thm}[\cite{iyamamutation}]
\label{iyamamutation}
Let $\frobcat$ be a stably $2$-Calabi--Yau Frobenius category, and let $T=\bigoplus_{k=1}^nT_k\in\frobcat$ be a cluster-tilting object, decomposed into indecomposable summands. Choose a non-projective summand $T_k$, and a minimal left $\add(T/T_k)$-approximation $T_k\to X_k$ of $T_k$. Then this map is an admissible monomorphism, meaning it yields a short exact sequence
\[\begin{tikzcd}[column sep=20pt]
0\arrow{r}&T_k\arrow{r}&X_k\arrow{r}&T_k^*\arrow{r}&0,
\end{tikzcd}\]
in the exact structure of $\frobcat$. If the Gabriel quiver of $\Endalg{\frobcat}{T}$ has no loops or $2$-cycles incident with the vertex corresponding to $T_k$, then $T_k^*$ is indecomposable and $\mu_kT:=(T/T_k)\oplus T_k^*$ is again cluster-tilting. One can also compute $T_k^*$ as the kernel of a minimal right $\add(T/T_k)$-approximation of $T_k$, which is necessarily an admissible epimorphism.
\end{thm}

Under the assumptions of \Cref{iyamamutation}, if $\Endalg{\frobcat}{T}\cong\frjac{Q}{F}{W}$ is a frozen Jacobian algebra, with Gabriel quiver $Q$, we would like the categorical mutation of $T$ at $T_k$ to be compatible with the combinatorial mutation of $(Q,F,W)$ at the vertex $k$ corresponding to this indecomposable summand. Precisely, we want an isomorphism
\[\Endalg{\frobcat}{\mu_k{T}}\cong\frjac{\mu_k Q}{\mu_k F}{\mu_k{W}}.\]
In this section, we explain conditions on $\frobcat$, $T$ and $T_k$ which ensure this compatibility, by generalising results of Buan, Iyama, Reiten and Smith \cite{buanmutation}, who provide an analogous theory in triangulated categories. Our arguments and exposition follow \cite{buanmutation} closely, with modifications as necessary to handle the frozen arrows and vertices appropriately.

This result has applications to cluster categorification. It is employed by the author in \cite{presslandcategorification} to show that mutation of cluster-tilting objects in categorifications of polarised principal coefficient cluster algebras, constructed in loc.\ cit., is compatible with Fomin--Zelevinsky mutation of quivers. We will explain at the end of the section how to deduce the corresponding result for the Grassmannian cluster categories of Jensen, King and Su \cite{jensencategorification}, for which it was not previously known.

Let $(Q,F,W)$ be a quiver with potential, and let $k\in Q_0^{\mut}$. Let $(Q',F,W'):=\lift{\mu}_k(Q,F,W)$ be the ice quiver with potential obtained after step \ref{QPm3} of the calculation of the mutation $(\mu_kQ,\mu_kF,\mu_kW)$, i.e.\ before the reduction step. By Theorem~\ref{thm:reduction}, there is an isomorphism
\[\frjac{Q'}{F}{W'}\cong\frjac{\mu_kQ}{\mu_kF}{\mu_kW},\]
but the former description of the algebra is more useful for the homological arguments in this section. The reader should note, however, that one may need to pass to the reduction in order to carry out iterated mutations, since this step can remove $2$-cycles that would otherwise prohibit mutations at their vertices.

For each arrow $\alpha\in Q_1$, there is an operation $\rightder{\alpha}{}\colon\cpa{\KK}{Q}\to\cpa{\KK}{Q}$ of right differentiation, defined on paths by
\[\rightder{\alpha}{\alpha_k\dotsc\alpha_1}=\begin{cases}\alpha_k\dotsm\alpha_2,&\alpha_1=\alpha,\\0,&\text{otherwise.}\end{cases}\]
There is also a left derivative $\leftder{\alpha}{}$, defined analogously. The main advantage of using $W'$ rather than $\mu_kW$ in this section is that we may use the more explicit description of $W'$ to compute the right derivatives of the relations it defines. For later use, we record these right derivatives, which are calculated directly from the definition, in the following lemma.

\begin{lem}[cf.~{\cite[Lem.~5.8]{buanmutation}}]
\label{rightder-calcs}
Let $(Q,F,W)$, $(Q',F,W')$ and $k$ be as above. Let $\alpha,\beta\in Q_1$ be arrows with $\tail{\alpha}=k=\head{\beta}$, and let $\gamma,\gamma'\in Q_1\cap Q_1'$. Then
\begin{enumerate}[label=(\roman*)]
\item\label{rightder1} $\rightder{\gamma}{\der{\gamma'}{W'}}=\rightder{\gamma}{\der{\gamma'}{W}}$,
\item\label{rightder2} $\rightder{\gamma}{\der{[\alpha\beta]}{W'}}=\rightder{\gamma}{\der{[\alpha\beta]}{[W]}}=\rightder{\gamma}{\rightder{\alpha}{\der{\beta}{W}}}$ and $\rightder{[\alpha\beta]}{\der{\gamma}{W'}}=\rightder{[\alpha\beta]}{\der{\gamma}{[W]}}=\rightder{\alpha}{\rightder{\beta}{\der{\gamma}{W}}}$,
\item\label{rightder3} $\rightder{[\alpha\beta]}{\der{\alpha^*}{W'}}=\beta^*$,
\item\label{rightder4} $\rightder{\beta^*}{\der{[\alpha\beta]}{W'}}=\alpha^*$,
\item\label{rightder5} $\rightder{\alpha^*}{\der{\beta^*}{W'}}=[\alpha\beta]$, and
\item\label{rightder6} For any $\delta,\delta'\in Q_1'$ such that $\rightder{\delta}{\der{\delta}{W'}}$ was not calculated in \ref{rightder1}--\ref{rightder5}, we have $\rightder{\delta}{\der{\delta}{W'}}=0$.\qed
\end{enumerate}
\end{lem}

Given an additive category $\cat$, and objects $X,Y\in\cat$, let $\catrad_\cat(X,Y)$ denote the subspace of $\Hom_\cat(X,Y)$ consisting of maps $f$ such that $\id{X}-gf$ is invertible for all $g\colon Y\to X$. We then define $\catrad^m_\cat(X,Y)$ to be the subspace of $\Hom_\frobcat(X,Y)$ consisting of maps that may be written as a composition $f_m\circ\dotsb\circ f_1$ with $f_i\in\catrad_\cat(X_{i-1},X_i)$ for some $X_i\in\cat$ (so that necessarily $X_0=X$ and $X_{m}=Y$). We extend this notation by
\[\catrad^0_\cat(X,Y):=\Hom_\cat(X,Y).\]
Note that, for any $m$, the subspace $\catrad^m_\cat(X,X)$ is an ideal of $\Endalg{\cat}{X}$. Moreover, if $\subcat\subseteq\cat$ is a full subcategory, then $\catrad^m_\subcat(X,Y)=\catrad^m_\cat(X,Y)$ if $m=0$ or $m=1$, but this equality need not hold if $m>2$.
More information about the radical of a category may be found in \cite[\S A.3]{assemelements}.

We will consider $\KK$-linear categories $\cat$ satisfying the conditions
\begin{enumerate}[label=({C}\arabic*)]
\item\label{C1} $\cat$ is Krull--Schmidt, and
\item\label{C2} for any non-zero basic object $X\in\cat$, we have
\begin{enumerate}[label=({A}\arabic*)]
\item\label{A1}$\Endalg{\cat}{X}/\catrad_\cat(X,X)\iso\KK^n$ for some $n>0$, and
\item\label{A2}$\Endalg{\cat}{X}\iso\varprojlim_{m\geq0}\Endalg{\cat}{X}/\catrad^m_\cat(X,X)$.
\end{enumerate}
\end{enumerate}
For example, if $B$ is a finite-dimensional Iwanaga--Gorenstein algebra, then the category
\[\GP(B)=\{X\in\fgmod{B}:\text{$\Ext^i(X,B)=0$ for all $i>0$}\}\]
of Gorenstein projective $B$-modules is a Frobenius category satisfying \ref{C1} and \ref{C2}; indeed, \ref{C2} is satisfied by any Hom-finite $\KK$-linear category.

Let $\cat$ be a category satisfying \ref{C1} and \ref{C2}, and let $Q$ be a finite quiver. For each vertex $i\in Q_0$, choose an indecomposable object $T_i\in\cat$, and for each arrow $a\colon i\to j$ in $Q$, choose a morphism $\Phi a\in\Hom_\cat(T_j,T_i)$. This data is equivalent to specifying an algebra homomorphism
\[\Phi\colon\cpa{\KK}{Q}\to\Endalg{\cat}{T},\]
where $T=\bigdsum_{i\in Q_0}T_i$ \cite[Lem.~3.5]{buanmutation}, with $\Phi(e_i)=\id{T_i}$ for each vertex idempotent $e_i$. Let $R$ be a finite subset of the closed ideal of $\cpa{\KK}{Q}$ generated by arrows, such that each $r\in R$ is basic, meaning it is a formal linear combination of paths of $Q$ with the same head and tail, and let $I=\close{\Span{R}}\leq\cpa{\KK}{Q}$. For example, the set of cyclic derivatives of a potential on $Q$ is a set of basic elements. Buan--Iyama--Reiten--Smith \cite[Prop.~3.6]{buanmutation} characterise when the homomorphism $\Phi$ above induces an isomorphism $\Phi\colon\cpa{\KK}{Q}/I\isoto\Endalg{\cat}{T}$ in terms of certain complexes in $\add{T}$ being right $2$-almost split, a definition we now recall.

\begin{defn}[{\cite[Def.~4.4]{buanmutation}}]
Let $\cat$ be an additive category, and let $T\in\cat$ be any object. Let
\[\begin{tikzcd}[column sep=20pt]U_1\arrow{r}{f_1}&U_0\arrow{r}{f_0}&X\end{tikzcd}\]
be a complex in $\add{T}$ such that $f_0$ is not a split epimorphism, and consider the induced sequence
\[\begin{tikzcd}[column sep=20pt]\Hom_\cat(T,U_1)\arrow{r}&\Hom_\cat(T,U_0)\arrow{r}&\catrad_\cat(T,X)\arrow{r}&0.\end{tikzcd}\]
We say that $f_0$ is \emph{right almost split} in $\add{T}$ if this induced sequence is exact at $\catrad_\cat(T,X)$, that $f_1$ is a \emph{pseudo-kernel} of $f_0$ in $\add{T}$ if this induced sequence is exact at $\Hom_\cat(T,U_0)$, and that the sequence $(f_1,f_0)$ is \emph{right $2$-almost split} if both of these conditions hold simultaneously.

We define \emph{left almost split} maps, \emph{pseudo-cokernels} and \emph{left $2$-almost split} sequences in $\add{T}$ dually, using the contravariant functor $\Hom_\cat(-,T)$, and call a complex
\[\begin{tikzcd}[column sep=20pt]Y\arrow{r}{f_2}&U_1\arrow{r}{f_1}&U_0\arrow{r}{f_0}&X\end{tikzcd}\]
\emph{weak $2$-almost split} in $\add{T}$ if $(f_1,f_0)$ is a right $2$-almost split sequence in $\add{T}$ and $(f_2,f_1)$ is a left $2$-almost split sequence in $\add{T}$.
\end{defn}

To establish our isomorphisms, we will use \cite[Prop.~3.3]{buanmutation} (see also \cite[Prop.~3.6]{buanmutation}, which is the same result in more categorical language). The following statement specialises this proposition to the case of frozen Jacobian algebras.

\begin{prop}[cf.~{\cite[Prop.~3.3]{buanmutation}}]
\label{birs-iso-prop}
Let $(Q,F,W)$ be an ice quiver with potential, $\cat$ an additive category satisfying \ref{C1} and \ref{C2}, and $\Phi\colon\cpa{\KK}{Q}\to\Endalg{\cat}{T}$ an algebra homomorphism. Write $T_i=\Phi(e_i)(T)$. Then the following are equivalent:
\begin{enumerate}[label=(\roman*)]
\item $\Phi$ induces an isomorphism $\frjac{Q}{F}{W}\isoto\Endalg{\cat}{T}$,
\item for every $i\in Q_0$, the complex
\begin{equation}
\label{2ras}
\begin{tikzcd}[column sep=40pt]
\displaystyle\bigdsum_{\substack{b\in Q_1^\mut\\\head{b}=i}}T_{\tail{b}}\arrow{r}{\Phi\rightder{a}{\der{b}{W}}}&\displaystyle\bigdsum_{\substack{a\in Q_1\\\tail{a}=i}}T_{\head{a}}\arrow{r}{\Phi a}&T_i
\end{tikzcd}
\end{equation}
is right $2$-almost split in $\add{T}$, and
\item for every $i\in Q_0$, the complex
\begin{equation}
\label{2las}
\begin{tikzcd}[column sep=40pt]
T_i\arrow{r}{\Phi b}&\displaystyle\bigdsum_{\substack{b\in Q_1\\\head{b}=i}}T_{\tail{b}}\arrow{r}{\Phi\leftder{b}{\der{a}{W}}}&\displaystyle\bigdsum_{\substack{a\in Q_1^\mut\\\tail{a}=i}}T_{\head{a}}
\end{tikzcd}
\end{equation}
is left $2$-almost split in $\add{T}$.
\end{enumerate}
\end{prop}

\begin{rem}
If $i$ is a mutable vertex, then the sequences \eqref{2ras} and \eqref{2las} glue together into a weak $2$-almost split sequence in $\add{T}$ with both outer terms given by $T_i$; see \cite[Lem.~4.1]{buanmutation} for the equality
\[\rightder{a}{\der{b}{W}}=\leftder{b}{\der{a}{W}}.\]
Thus in the context of \cite[\S5]{buanmutation}, which deals with ordinary Jacobian algebras, it is both possible and convenient to phrase assumptions and conclusions in terms of the existence of such weak $2$-almost split sequences, even though one then proves more than is strictly necessary to obtain the mutation results. Since this symmetry breaks down at frozen vertices, we must make a choice, and we choose to use right $2$-almost split sequences in these cases. Indeed, it is this breaking of symmetry that results in the main differences between our arguments and those of \cite{buanmutation}; we have to pick out which of the two dual arguments provided by loc.\ cit.\ applies at each step of our proof.
\end{rem}

Under the notation and assumptions of \Cref{birs-iso-prop}, let $k\in Q_0^\mut$ be a mutable vertex. Let $T_k^*\in\cat$ be an indecomposable object not in $\add{T}$, and write $\mu_kT=T/T_k\dsum T_k^*$. We make the following assumptions, labelled for consistency with the corresponding assumptions of \cite[\S5.2]{buanmutation}. Our assumptions differ from these only by additional conditions at frozen vertices in \ref{O} and \ref{IV}, and conventions on composing maps.

\begin{enumerate}[label=(\Roman*)]
\myitem[(O)]\label{O} The map $\Phi$ induces an isomorphism $\frjac{Q}{F}{W}\isoto\Endalg{\cat}{T}$. By \Cref{birs-iso-prop}, this condition may be phrased equivalently as follows: for every $i\in Q_0^\mut$, the complex
\[\begin{tikzcd}[column sep=40pt]
T_i\arrow{r}{\Phi b}&\displaystyle\bigdsum_{\substack{b\in Q_1\\\head{b}=i}}T_{\tail{b}}\arrow{r}{\Phi\rightder{a}{\der{b}{W}}}&\displaystyle\bigdsum_{\substack{a\in Q_1\\\tail{a}=i}}T_{\head{a}}\arrow{r}{\Phi a}&T_i
\end{tikzcd}\]
is a weak $2$-almost split sequence in $\add{T}$, which we abbreviate to
\[\begin{tikzcd}[column sep=20pt]
T_i\arrow{r}{f_{i2}}&U_{i1}\arrow{r}{f_{i1}}&U_{i0}\arrow{r}{f_{i0}}&T_i,
\end{tikzcd}\]
and for each $i\in F_0$, the complex
\[\begin{tikzcd}[column sep=40pt]
\displaystyle\bigdsum_{\substack{b\in Q_1^m\\\head{b}=i}}T_{\tail{b}}\arrow{r}{\Phi\rightder{a}{\der{b}{W}}}&\displaystyle\bigdsum_{\substack{a\in Q_1\\\tail{a}=i}}T_{\head{a}}\arrow{r}{\Phi a}&T_i
\end{tikzcd}\]
is a right $2$-almost split sequence in $\add{T}$, which we abbreviate to
\[\begin{tikzcd}[column sep=20pt]
U_{i1}\arrow{r}{f_{i1}}&U_{i0}\arrow{r}{f_{i0}}&T_i.
\end{tikzcd}\]
\setcounter{enumi}{0}
\item\label{I} There exist complexes
\[
\begin{tikzcd}[column sep=20pt]
T_k\arrow{r}{f_{k2}}&U_{k1}\arrow{r}{h_k}&T_k^*,\end{tikzcd}\qquad
\begin{tikzcd}[column sep=20pt]T_k^*\arrow{r}{g_k}&U_{k0}\arrow{r}{f_{k0}}&T_k
\end{tikzcd}
\]
in $\cat$ such that $f_{k1}=g_kh_k$.
\item\label{II} The complex
\[\begin{tikzcd}[column sep=35pt]
T_k^*\arrow{r}{g_k}&U_{k0}\arrow{r}{f_{k0}f_{k2}}&U_{k1}\arrow{r}{h_k}&T_k^*
\end{tikzcd}\]
is a weak $2$-almost split sequence in $\add(\mu_kT)$.
\item\label{III} The sequences
\[\begin{tikzcd}[column sep=30pt,row sep=3pt]
\Hom_\cat(T_k^*,T_k^*)\arrow{r}{h_k}&\Hom_\cat(U_{k1},T_k^*)\arrow{r}{f_{k2}}&\Hom_\cat(T_k,T_k^*),\\
\Hom_\cat(T_k^*,T_k^*)\arrow{r}{g_k}&\Hom_\cat(T_k^*,U_{k0})\arrow{r}{f_{k0}}&\Hom_\cat(T_k^*,T_k),
\end{tikzcd}\]
obtained from those of \ref{I} by applying $\Hom_\cat(-,T_k^*)$ and $\Hom_\cat(T_k^*,-)$ respectively, are exact.
\item\label{IV} For all $i\in Q_0$, we have $T_k\notin(\add{U_{i1}})\cap(\add{U_{i0}})$; equivalently there are no $2$-cycles of $Q$ incident with $k$. For $i\in Q_0^\mut$ the sequences
\[\begin{tikzcd}[column sep=30pt,row sep=3pt]
\Hom_\cat(T_k^*,U_{i1})\arrow{r}{f_{i1}}&\Hom_\cat(T_k^*,U_{i0})\arrow{r}{f_{i0}}&\Hom_\cat(T_k^*,T_i),\\
\Hom_\cat(U_{i0},T_k^*)\arrow{r}{f_{i1}}&\Hom_\cat(U_{i1},T_k^*)\arrow{r}{f_{i2}}&\Hom_\cat(T_i,T_k^*),
\end{tikzcd}\]
obtained by applying $\Hom_\cat(T_k^*,-)$ and $\Hom_\cat(-,T_k^*)$ respectively to the weak $2$-almost split sequence from \ref{O}, are exact. For each $i\in F_0$, we have an exact sequence
\[\begin{tikzcd}[column sep=30pt,row sep=3pt]
\Hom_\cat(T_k^*,U_{i1})\arrow{r}{f_{i1}}&\Hom_\cat(T_k^*,U_{i0})\arrow{r}{f_{i0}}&\Hom_\cat(T_k^*,T_i),
\end{tikzcd}\]
obtained by applying $\Hom_\cat(T_k^*,-)$ to the right $2$-almost split sequence from \ref{O}.
\end{enumerate}

\begin{lem}
\label{assumption-knitting}
Let $\frobcat$ be a stably $2$-Calabi--Yau Frobenius category satisfying \ref{C1} and \ref{C2}, let $(Q,F,W)$ be an ice quiver with potential, and let $T\in\frobcat$ be a cluster-tilting object. Let $\Phi\colon\cpa{\KK}{Q}\to\Endalg{\frobcat}{T}$ be an algebra homomorphism inducing an isomorphism
\[\Phi\colon\frjac{Q}{F}{W}\isoto\Endalg{\frobcat}{T}.\]
If $k\in Q_0^\mut$ is not incident with any loops or $2$-cycles, and $\Phi(e_k)(T)$ is not projective, then there exists $T_k^*\not\in\add{T}$ such that $\Phi$, $T$ and $T_k^*$ satisfy the assumptions \ref{O}--\ref{IV}.
\end{lem}
\begin{proof}
By the assumptions on $\Phi$, we have that $T_k$ is an indecomposable non-projective summand of the cluster-tilting object $T$. Since $\frobcat$ is stably $2$-Calabi--Yau and $Q$ has no loops or $2$-cycles at $k$, we may take $T_k^*$ as in \Cref{iyamamutation}. For this choice of $T_k^*$, most of our desired statements are proved in \cite[Lem.~5.7]{buanmutation}. Note in particular that the complexes in \ref{I} are in fact the short exact sequences from \Cref{iyamamutation}; we will use this below. It remains to check the statements of \ref{O} and \ref{IV} dealing with frozen vertices.

The existence of the required right $2$-almost split sequence in \ref{O} follows from the statement (i)$\implies$(ii) of \Cref{birs-iso-prop}. Since there are no $2$-cycles of $Q$ incident with $k$, the statement that $T_k\notin(\add{U_{i1}})\cap(\add{U_{i0}})$ holds when $i$ is frozen exactly as when $i$ is unfrozen. For the remaining statement in \ref{IV}, consider the diagram
\begin{equation}
\label{IV-frozen-diag}
\begin{tikzcd}[column sep=20pt]
0\arrow{d}&0\arrow{d}&0\arrow{d}\\
\Hom_\frobcat(T_k,U_{i1})\arrow{r}\arrow{d}&\Hom_\frobcat(T_k,U_{i0})\arrow{r}\arrow{d}&\Hom_\frobcat(T_k,T_i)\arrow{r}\arrow{d}&0\\
\Hom_\frobcat(U_{k0},U_{i1})\arrow{r}\arrow{d}&\Hom_\frobcat(U_{k0},U_{i0})\arrow{r}\arrow{d}&\Hom_\frobcat(U_{k0},T_i)\arrow{d}\\
\Hom_\frobcat(T_k^*,U_{i1})\arrow{r}\arrow{d}&\Hom_\frobcat(T_k^*,U_{i0})\arrow{r}\arrow{d}&\Hom_\frobcat(T_k^*,T_i)\arrow{d}\\
0&0&0
\end{tikzcd}
\end{equation}
in which the lowest non-zero row is the sequence we wish to prove is exact. The columns are obtained by applying $\Hom_\frobcat(-,X)$ to the short exact sequence
\[\begin{tikzcd}[column sep=20pt]
0\arrow{r}&T_k^*\arrow{r}{g_k}&U_{k0}\arrow{r}{f_{k0}}&T_k\arrow{r}&0
\end{tikzcd}\]
for various $X\in\add{T}$; since $T$ is cluster-tilting, we have $\Ext^1_\frobcat(T_k,X)=0$ in each case, and so these columns are short exact sequences. The rows are obtained by applying $\Hom_\frobcat(Y,-)$ to the complex
\[\begin{tikzcd}[column sep=20pt]
U_{i1}\arrow{r}{f_{i1}}&U_{i0}\arrow{r}{f_{i0}}&T_i,
\end{tikzcd}\]
which we have already shown is right $2$-almost split in $\add{T}$, for various $Y\in\frobcat$. In the case of the first two rows, we even take $Y\in\add{T}$; it then follows immediately from the definition of right $2$-almost splitness that the second row is exact. Exactness of the first row follows similarly, using that $T_k\not\cong T_i$ to see that
\[\Hom_\frobcat(T_k,T_i)=\catrad_\frobcat(T_k,T_i),\]
so that we also have exactness at $\Hom_\frobcat(T_k,T_i)$. Exactness of the lowest row now follows by viewing the diagram \eqref{IV-frozen-diag} as a short exact sequence of chain complexes, and passing to the long-exact sequence in cohomology.
\end{proof}

\begin{eg}
\label{applications}
We pick out three families of Frobenius cluster categories for which some cluster-tilting objects have endomorphism algebra isomorphic to $\frjac{Q}{F}{W}$ for $Q$ without loops and $2$-cycles, so we can apply \Cref{assumption-knitting}. For cases \ref{eg-BIRS} and \ref{eg-JKS}, proofs that the categories are indeed Frobenius cluster categories can be found in \cite[Eg.~3.11--12]{presslandinternally}. In case \ref{eg-P}, this is part of \cite[Thm.~1]{presslandcategorification}. The fact that the relevant quivers have no loops or $2$-cycles is a direct consequence of the explicit construction in each case.
\begin{enumerate}[label=(\roman*)]
\item\label{eg-BIRS} Buan--Iyama--Reiten--Scott \cite{buancluster} associate Frobenius cluster categories $\cat_w$ to elements $w$ of Coxeter groups. Each reduced expression $\mathbf{i}$ for $w$ in terms of simple reflections determines a cluster-tilting object $T_{\mathbf{i}}\in\cat_w$, and Buan--Iyama--Reiten--Smith have shown that $\Endalg{\cat_w}{T_{\mathbf{i}}}$ is isomorphic to a frozen Jacobian algebra determined by $\mathbf{i}$ \cite[Thm.~6.6]{buanmutation}.
\item\label{eg-P} In \cite{presslandcategorification} the author constructs, for any acyclic quiver $Q$, a Frobenius cluster category $\GP(B_Q)$, of Gorenstein projective modules over an Iwanaga--Gorenstein algebra $B_Q$, containing a cluster-tilting object with endomorphism algebra isomorphic to a frozen Jacobian algebra constructed explicitly from $Q$ \cite[Thm.~5.3]{presslandcategorification}.
\item\label{eg-JKS} Jensen--King--Su \cite{jensencategorification} describe a Frobenius cluster category $\CM(B_{k,n})$, consisting of Cohen--Macaulay modules over an algebra $B_{k,n}$, categorifying Scott's cluster algebra structure on the homogeneous coordinate ring of the Grassmannian $G_k^n$ of $k$-dimensional subspaces of $\CC^n$ \cite{scottgrassmannians}. A $(k,n)$-Postnikov diagram $D$ determines both a cluster of Plücker coordinates in the cluster algebra, and a cluster-tilting object $T_D\in\CM(B_{k,n})$. Baur--King--Marsh show that $\Endalg{B_{k,n}}{T_D}$ is isomorphic to a frozen Jacobian algebra determined by $D$ \cite[Thm.~10.3]{baurdimer}. This algebra may also be realised \cite[\S2]{baurdimer} as the dimer algebra of a dimer model in the disk, as in Example~\ref{e:dimers2}.
\end{enumerate}
These categories also satisfy \ref{C1} (this being part of the definition of a Frobenius cluster category) and \ref{C2}, providing that one uses the complete version of $B_{k,n}$ in \ref{eg-JKS} (cf.\ \cite[Rem.~3.3]{jensencategorification}). For example, this makes the endomorphism algebra of any basic object of $\CM(B_{k,n})$ a finitely generated $\powser{\CC}{t}$-module, so that \ref{A1} and \ref{A2} hold.
\end{eg}

Under the notation and assumptions of \Cref{birs-iso-prop}, let $k\in Q_0^\mut$ be a mutable vertex. Choose $T_k^*\notin\add{T}$ and write $\mu_kT=T/T_k\dsum T_k^*$. Assume \ref{O}--\ref{IV}. By \ref{IV}, there are no $2$-cycles in $Q$ incident with $k$, so we may take $(Q',F,W')=\lift{\mu}_k(Q,F,W)$; note that this is not the ordinary mutation of $(Q,F,W)$ since we do not perform the final reduction step, in order to give us better control over the arrows of $Q'$, but it defines the same frozen Jacobian algebra as the ordinary mutation. We now define an algebra homomorphism $\Phi'\colon\cpa{\KK}{Q'}\to\Endalg{\cat}{\mu_kT}$ by choosing a summand of $\mu_kT$ for each $i\in Q_0'=Q_0$ and a map $\Phi'a\colon T_j\to T_i$ for each arrow $a\colon i\to j$ in $Q'_1$, as follows. For $i\ne k$, we associate $T_i$ to $i$, exactly as for $T$, and complete the assignment of summands to vertices by associating the new summand $T_k^*$ to $k$. On arrows, we define $\Phi'$ as follows.
\begin{enumerate}[label=(\roman*)]
\item If $a$ is an arrow common to $Q$ and $Q'$, then we take $\Phi'a=\Phi a$.
\item On arrows $[ab]$ of $Q'$, define $\Phi'[ab]=\Phi b\circ\Phi a$.
\item Recall that by assumption \ref{I} we have maps
\[g_k\colon T_k^*\to\bigdsum_{\substack{a\in Q_1\\\tail{a}=k}}T_{\head{a}},\qquad h_k\colon\bigdsum_{\substack{b\in Q_1\\\head{b}=k}}T_{\tail{b}}\to T_k^*.\]
If $a\in Q_1$ has $\tail{a}=k$, define $\Phi'a^*$ to be the component of $g_k$ indexed by $a$, and if $b\in Q_1$ has $\head{b}=k$, define $\Phi'b^*$ to be the component of $-h_k$ indexed by $b$.
\end{enumerate}
We are now able to state the main result of this section.

\begin{thm}
\label{mutation-thm}
Let $\frobcat$ be a stably $2$-Calabi--Yau Frobenius category satisfying \ref{C1}--\ref{C2}, let $T\in\frobcat$ be a cluster-tilting object, and assume we have an isomorphism $\Phi\colon\frjac{Q}{F}{W}\isoto\Endalg{\frobcat}{T}$ for some ice quiver with potential $(Q,F,W)$. If $k\in Q_0^\mut$ is a mutable vertex of $Q$ not incident with loops or $2$-cycles, and $\Phi(e_k)(T)$ is not projective, then there is an indecomposable object $T_k^*$ in $\frobcat$, unique up to isomorphism, such that $T_k^*\not\iso T_k$ and $\mu_kT=T/T_k\oplus T_k^*$ is cluster-tilting, and an isomorphism $\frjac{\mu_kQ}{\mu_kF}{\mu_kW}\isoto\Endalg{\frobcat}{\mu_kT}$, induced from the map $\Phi'$ constructed above.
\end{thm}
\begin{proof}
The existence and uniqueness of $T_k^*$ follows from \Cref{iyamamutation} and \Cref{assumption-knitting}, so we need only find the necessary isomorphism. Writing $(Q',F,W')=\lift{\mu}_k(Q,F,W)$, the results of \Cref{s:reduction} give us an isomorphism $\frjac{\mu_kQ}{\mu_kF}{\mu_kW}\isoto\frjac{Q'}{F}{W'}$, so it is enough to show that the map $\Phi'$ constructed above induces an isomorphism $\frjac{Q'}{F}{W'}\isoto\Endalg{\frobcat}{\mu_kT}$.

To do this, we will apply the statement (ii)$\implies$(i) of \Cref{birs-iso-prop}, so it suffices to show, for each $i\in Q_0'$, that the sequence
\begin{equation}
\label{ras'}
\begin{tikzcd}[column sep=45pt]
\displaystyle\bigdsum_{\substack{d\in Q_1'^\mut\\\head{d}=i}}T_{\tail{d}}\arrow{r}{\Phi'\rightder{c}{\der{d}{W'}}}&\displaystyle\bigdsum_{\substack{c\in Q_1'\\\tail{c}=i}}T_{\head{c}}\arrow{r}{\Phi'c}&T_i
\end{tikzcd}
\end{equation}
is right $2$-almost split in $\add(\mu_kT)$. When $i$ is mutable, this follows from \cite[Thm.~5.6]{buanmutation}, so we need only deal with the case $i\in F_0$. Our argument follows closely the proof of \cite[Lem.~5.10]{buanmutation}, using freely computations of the derivatives $\rightder{c}{\der{d}{W'}}$ from \Cref{rightder-calcs}. We treat elements of direct sums as column vectors, with maps acting as matrices from the left; this convention is transposed from that of \cite{buanmutation}.

Let $i\in F_0$. Since $Q$ has no $2$-cycles incident with $k$, either there is no arrow $k\to i$ in $Q$, or there is no arrow $i\to k$ in $Q$. In the first case, the sequence \eqref{ras'} has the form
\begin{equation}
\label{ras'-case1}
\begin{tikzcd}[column sep=50pt,ampersand replacement=\&]
\begin{array}{c}\Big(\displaystyle\bigdsum_{\substack{b\in Q_1\\b\colon i\to k}}T_k^*\Big)\\\dsum\\\Big(\displaystyle\bigdsum_{\substack{d\in Q_1^\mut\\\head{d}=i}}T_{\tail{d}}\Big)\end{array}\arrow{r}{x}\&\begin{array}{c}\Big(\displaystyle\bigdsum_{\substack{a,b\in Q_1\\\tail{a}=k\\b\colon i\to k}}T_{\head{a}}\Big)\\\dsum\\\Big(\displaystyle\bigdsum_{\substack{c\in Q_1\\\head{c}\ne k\\\tail{c}=i}}T_{\head{c}}\Big)\end{array}\arrow{r}{(\begin{smallmatrix}\Phi'[ab]&\Phi'c\end{smallmatrix})}\&T_i,
\end{tikzcd}
\end{equation}
where the direct sums are divided so that the upper portion consists of the contribution from arrows in $Q_1'\setminus Q_1$, and $x$ is given by the matrix
\[x=\begin{pmatrix}\Phi'a^*&\Phi'\rightder{[ab]}{\der{d}{[W]}}\\0&\Phi'\rightder{c}{\der{d}{[W]}}\end{pmatrix}.\]
First we check that this is a complex, by computing
\begin{gather*}
\sum_{\substack{a\in Q_1\\\tail{a}=k}}\Phi'[ab]\Phi'a^*=\Phi bf_{k0}g_k=0,\\
\sum_{\substack{a,b\in Q_1\\\tail{a}=k\\b\colon i\to k}}\Phi'[ab]\Phi'\rightder{[ab]}{\der{d}{[W]}}+\sum_{\substack{c\in Q_1\\\head{c}\ne k\\\tail{c}=i}}\Phi'c\Phi'\rightder{c}{\der{d}{[W]}}=\Phi\der{d}{W}=0
\end{gather*}
for each $b\colon i\to k$ in $Q_1$ and $d\in Q_1^\mut$ with $\head{d}=i$. Let $\ell$ be the number of arrows $i\to k$ in $Q$. Then we have $U_{i0}=T_k^\ell\dsum U_{i0}''$ with $T_k\notin\add{U_{i0}''}$, and the maps $f_{i0}$ and $f_{i1}$ from the right $2$-almost split sequence of \ref{O} decompose as
\[
f_{i0}=\begin{pmatrix}f_{i0}'&f_{i0}''\end{pmatrix}\colon T_k^\ell\dsum U_{i0}''\to T_i,\qquad
f_{i1}=\begin{pmatrix}f_{i1}'\\f_{i1}''\end{pmatrix}\colon U_{i1}\to T_k^\ell\dsum U_{i0}.
\]
We may then rewrite \eqref{ras'-case1} as
\[\begin{tikzcd}[column sep=50pt,ampersand replacement=\&]
\begin{array}{c}T_k^{*\ell}\\\dsum\\U_{i1}\end{array}\arrow{r}{\left(\begin{smallmatrix}g_k^\ell&t\\0&f_{i1}''\end{smallmatrix}\right)}\&\begin{array}{c}U_{k0}^\ell\\\dsum\\U_{i0}''\end{array}\arrow{r}{(\begin{smallmatrix}f_{i0}'f_{k0}^\ell&f_{i0}''\end{smallmatrix})}\&T_i,
\end{tikzcd}\]
where $f_{k0}^\ell t=f_{i1}'$.

Next we show that $(\begin{smallmatrix}f_{i0}'f_{k0}^\ell&f_{i0}''\end{smallmatrix})$ is right almost split in $\add{\mu_k T}$. Let $p\in\catrad_\cat(T/T_k,T_i)$. Since $f_{i0}=(\begin{smallmatrix}f_{i0}'&f_{i0}''\end{smallmatrix})$ is right almost split in $\add{T}$, there exists $(\begin{smallmatrix}p_1\\p_2\end{smallmatrix})\colon T/T_k\to T_k^\ell\dsum U_{i0}''$ such that $p=f_{i0}'p_1+f_{i0}''p_2$. Moreover, since $f_{k0}$ is right almost split in $\add{T}$, there exists $q\colon T/T_k\to U_{k0}^\ell$ such that $p_1=f_{k0}^\ell q$, and so
\[p=f_{i0}'f_{k0}^\ell q+f_{i0}''p_2\]
factors through $(\begin{smallmatrix}f_{i0}'f_{k0}^\ell&f_{i0}''\end{smallmatrix})$ as required. On the other hand, if $p\in\catrad_\cat(T_k^*,T_i)$, then since $g_k$ is left almost split in $\add(\mu_kT)$ there exists $q\colon U_{k0}\to T_i$ such that $p=qg_k$. Since there are no arrows $k\to i$ in $Q$, there are no summands of $U_{k0}$ isomorphic to $T_i$, and so $q\in\catrad_\cat(U_{k0},T_i)$. Since $U_{k0}\in\add(T/T_k)$, we see as above that $q$, and therefore $p$, factors through $(\begin{smallmatrix}f_{i0}'f_{k0}^\ell&f_{i0}''\end{smallmatrix})$.

Now we show that $\left(\begin{smallmatrix}g_k^\ell&t\\0&f_{i1}''\end{smallmatrix}\right)$ is a pseudo-kernel of $(\begin{smallmatrix}f_{i0}'f_{k0}^\ell&f_{i0}''\end{smallmatrix})$ in $\add{\mu_k T}$. By \ref{III} and \ref{IV} we have exact sequences
\begin{equation}
\label{13'}
\begin{tikzcd}[column sep=40pt]
\Hom_\cat(\mu_kT,T_k^*)\arrow{r}{g_k}&\Hom_\cat(\mu_kT,U_{k0})\arrow{r}{f_{k0}}&\Hom_\cat(\mu_kT,T_k)
\end{tikzcd}
\end{equation}
and
\begin{equation}
\label{14'}
\begin{tikzcd}[column sep=37pt,ampersand replacement=\&]
\Hom_\cat(\mu_kT,U_{i1})\arrow{r}{\left(\begin{smallmatrix}f_{i1}'\\f_{i1}''\end{smallmatrix}\right)}\&\Hom_\cat(\mu_kT,T_k^\ell\dsum U_{i0}'')\arrow{r}{(\begin{smallmatrix}f_{i0}'&f_{i0}''\end{smallmatrix})}\&\Hom_\cat(\mu_kT,T_i).
\end{tikzcd}
\end{equation}
Now if $\left(\begin{smallmatrix}p_1\\p_2\end{smallmatrix}\right)\colon\mu_kT\to U_{k0}^\ell\dsum U_{i0}''$ satisfies 
\[0=\begin{pmatrix}f_{i0}'f_{k0}^\ell&f_{i0}''\end{pmatrix}\begin{pmatrix}p_1\\p_2\end{pmatrix}
=\begin{pmatrix}f_{i0}'&f_{i0}''\end{pmatrix}\begin{pmatrix}f_{k0}^\ell&0\\0&1\end{pmatrix}\begin{pmatrix}p_1\\p_2\end{pmatrix},\]
then by exactness of \eqref{14'} there exists $q\colon\mu_kT\to U_{i1}$ such that
\[\begin{pmatrix}f_{i1'}\\f_{i1}''\end{pmatrix}q=\begin{pmatrix}f_{k0}^\ell&0\\0&1\end{pmatrix}\begin{pmatrix}p_1\\p_2\end{pmatrix}.\]
It follows that $f_{k0}^\ell p_1=f_{i1}'q$ and $p_2=f_{i1}''q$. In particular,
\[f_{k0}^\ell(p_1-tq)=f_{i1}'q-f_{i1}q=0,\]
so by exactness of \eqref{13'} there exists $r\colon\mu_kT\to T_k^{*\ell}$ such that $p_1-tq=g_k^\ell r$. It follows that
\[\begin{pmatrix}p_1\\p_2\end{pmatrix}=\begin{pmatrix}g_k^\ell&t\\0&f_{i1}''\end{pmatrix}\begin{pmatrix}r\\q\end{pmatrix},\]
completing the proof that \eqref{ras'} is right $2$-almost split when there are no arrows $k\to i$ in $Q$.

Now assume instead that there are no arrows $i\to k$ in $Q$. In this case, the sequence \eqref{ras'} has the form
\begin{equation}
\label{ras'-case2}
\begin{tikzcd}[column sep=50pt,ampersand replacement=\&]
\begin{array}{c}\Big(\displaystyle\bigdsum_{\substack{a,b\in Q_1\\\head{b}=k\\a\colon k\to i}}T_{\tail{b}}\Big)\\\dsum\\\Big(\displaystyle\bigdsum_{\substack{d\in Q_1^\mut\\\head{d}=i\\\tail{d}\ne k}}T_{\tail{d}}\Big)\end{array}\arrow{r}{y}\&\begin{array}{c}\Big(\displaystyle\bigdsum_{\substack{a\in Q_1\\a\colon k\to i}}T_k^*\Big)\\\dsum\\\Big(\displaystyle\bigdsum_{\substack{c\in Q_1\\\tail{c}=i}}T_{\head{c}}\Big)\end{array}\arrow{r}{(\begin{smallmatrix}\Phi'a^*&\Phi'c\end{smallmatrix})}\&T_i,
\end{tikzcd}
\end{equation}
where
\[y=\begin{pmatrix}\Phi' b^*&0\\\Phi'\rightder{[ab]}{\der{c}{[W]}}&\Phi'\rightder{d}{\der{c}{[W]}}\end{pmatrix}.\]
We see using \ref{I} that this is a complex, since
\begin{align*}
\sum_{\substack{a\in Q_1\\ a\colon k\to i}}\Phi'a^*\Phi'b^*+\sum_{\substack{c\in Q_1\\\tail{c}=i}}\Phi'c\Phi'\rightder{[ab]}{\der{c}{[W]}}&=(-g_kh_k+\Phi\rightder{a}{\der{b}{W}})|_{T_{\tail{b}}}^{T_i}\\
&=(-f_{k1}+f_{k1})|_{T_{\tail{b}}}^{T_i}=0,\\
\sum_{\substack{c\in Q_1\\\tail{c}=i}}\Phi'c\Phi'\rightder{d}{\der{c}{[W]}}&=\Phi\der{d}{W}=0
\end{align*}
for each pair $a,b\in Q_1$ with $\head{b}=k$ and $a\colon k\to i$, and each $d\in Q_1^\mut$ with $\head{d}=i$ and $\tail{d}\ne k$. (The notation after the first equality sign on the first line refers to taking the component $T_{\tail{b}}\to T_i=T_{\head{a}}$ indexed by the pair $(a,b)$.) Let $\ell$ be the number of arrows $k\to i$ in $Q$. Then $U_{i1}=T_k^\ell\dsum U_{i1}''$, where $T_k\notin\add{U_{i1}''}$, and $f_{i1}$ decomposes as
\[f_{i1}=\begin{pmatrix}f_{i1}'&f_{i1}''\end{pmatrix}\colon T_k^\ell\dsum U_{i1}''\to U_{i0}.\]
We may then rewrite \eqref{ras'-case2} as
\[\begin{tikzcd}[column sep=50pt,ampersand replacement=\&]
\begin{array}{c}U_{k1}^\ell\\\dsum\\U_{i1}''\end{array}\arrow{r}{\left(\begin{smallmatrix}-h_k^\ell&0\\s&f_{i1''}\end{smallmatrix}\right)}\&\begin{array}{c}T_k^{*\ell}\\\dsum\\U_{i0}\end{array}\arrow{r}{(\begin{smallmatrix}u&f_{i0}\end{smallmatrix})}\&T_i,
\end{tikzcd}\]
where $sf_{k2}^\ell=f_{i1}'$ and $f_{i0}s=uh_k^\ell$.

Before showing that this sequence is right $2$-almost split in $\add{\mu_kT}$, we establish that the map $u\colon T_k^\ell\to T_i$, whose components are given by $\Phi'a^*$ for the $\ell$ arrows $a\colon k\to i$, induces a bijection
\begin{equation}
\label{rad-bij}
u\colon\Hom_\cat(T_k^*,T_k^{*\ell})/\catrad_\cat(T_k^*,T_k^{*\ell})\isoto\catrad_\cat(T_k^*,T_i)/\catrad^2_{\add(\mu_kT)}(T_k^*,T_i).
\end{equation}
By \ref{C2}, we have $\Hom_\cat(T_k^*,T_k^*)/\catrad_\cat(T_k^*,T_k^*)\iso\KK$, spanned by the class of the identity, so it is sufficient to show that $\catrad_\cat(T_k^*,T_i)/\catrad^2_{\add(\mu_kT)}(T_k^*,T_i)$ has as basis the $\ell$ maps $\Phi'a^*$ for $a\colon k\to i$. These maps are some of the components of $g_k$, which is left almost split in $\add(\mu_kT)$ by \ref{II}, meaning that its components span $\catrad_\cat(T_k^*,U_{k0})/\catrad^2_{\add(\mu_kT)}(T_k^*,U_{k0})$. Since there is no $2$-cycle of $Q$ incident with $k$, we have $f_{k0}f_{k2}\in\catrad_\cat(U_{k0},U_{k1})$, from which it follows that $g_k$ is also left minimal, i.e.\ that its components are linearly independent in $\catrad_\cat(T_k^*,U_{k0})/\catrad^2_{\add(\mu_kT)}(T_k^*,U_{k0})$, hence a basis. Restricting to the summands of $U_{k0}$ isomorphic to $T_i$ then gives the desired result.

We may now show that $(\begin{smallmatrix}u&f_{i0}\end{smallmatrix})$ is right almost split in $\add(\mu_kT)$. Since $f_{i0}$ is right almost split in $\add{T}$ by \ref{O}, for any $p\in\catrad_\cat(T/T_k,T_i)$ there exists $p'\colon T/T_k\to U_{i0}$ such that
\[p=f_{i0}p'=\begin{pmatrix}u&f_{i0}\end{pmatrix}\begin{pmatrix}0\\p'\end{pmatrix}.\]
On the other hand, if $p\in\catrad_\cat(T_k^*,T_i)$, then by \eqref{rad-bij} there exists $p_1\in\Hom_\cat(T_k^*,T_k^{*\ell})$ such that $p-up_1\in\catrad^2_{\add(\mu_kT)}(T_k^*,T_i)$. Since $g_k$ is left almost split in $\add(\mu_kT)$, there exists $q\colon U_{k0}\to T_i$ such that $p-up_1=qg_k$. Now, using again that $f_{i0}$ is right almost split in $\add{T}$, there exists $r\colon U_{k0}\to U_{i0}$ such that $q=f_{i0}r$, so that
\[p=up_1+f_{i0}rg_k=\begin{pmatrix}u&f_{i0}\end{pmatrix}\begin{pmatrix}p_1\\rg_k\end{pmatrix}\]
factors through $(\begin{smallmatrix}u&f_{i0}\end{smallmatrix})$ as required.

Finally, we show that $\left(\begin{smallmatrix}-h_k^\ell&0\\s&f_{i1''}\end{smallmatrix}\right)$ is a pseudo-kernel of $(\begin{smallmatrix}u&f_{i0}\end{smallmatrix})$ in $\add(\mu_kT)$. Assume that $\left(\begin{smallmatrix}p_1\\p_2\end{smallmatrix}\right)\colon T'\to T_k^{*\ell}\dsum U_{i0}$ satisfies
\[\begin{pmatrix}u&f_{i0}\end{pmatrix}\begin{pmatrix}p_1\\p_2\end{pmatrix}=0.\]
To see that $p_1$ factors through $h_k^\ell$, we first show that $p_1\in\catrad_\cat(T',T_k^{*\ell})$, for which it suffices to consider the case $T'=T_k^*$. We then have
\[p_2\in\Hom_\cat(T_k^*,U_{i0})=\catrad_{\add(\mu_kT)}(T_k^*,U_{i0}),\]
and $f_{i0}\in\catrad_{\add{T}}(U_{i0},T_i)=\catrad_{\add(\mu_kT)}(U_{i0},T_i)$ by \ref{O} and the assumption that there are no arrows $i\to k$ in $Q$, so that $U_{i0}\in\add(\mu_kT)$. It follows that
\[up_1=-f_{i0}p_2\in\catrad^2_{\add(\mu_kT)}(T_k^*,T_i),\]
so by \eqref{rad-bij} we have $p_1\in\catrad_\cat(T_k^*,T_k^{*\ell})$ as required. Now since $h_k$ is right almost split in $\add(\mu_kT)$ by \ref{II}, there exists $q\colon T'\to U_{k1}^\ell$ such that $p_1=h_k^\ell q$.

By \ref{III} and \ref{IV} we have an exact sequence
\[\begin{tikzcd}[column sep=40pt,ampersand replacement=\&]
\Hom_\cat(\mu_kT,T_k^\ell\dsum U_{i0}'')\arrow{r}{(\begin{smallmatrix}f_{i1}'&f_{i1}''\end{smallmatrix})}\&\Hom_\cat(\mu_kT,U_{i0})\arrow{r}{f_{i0}}\&\Hom_\cat(\mu_kT,T_i).
\end{tikzcd}\]
Since $f_{i0}(p_2+sq)=f_{i0}p_2+uh_k^\ell q=0$, it follows that there exists $\left(\begin{smallmatrix}q_1\\q_2\end{smallmatrix}\right)\colon T'\to T_k^\ell\dsum U_{i0}$ such that
\[p_2+sq=\begin{pmatrix}f_{i1'}&f_{i1}''\end{pmatrix}\begin{pmatrix}q_1\\q_2\end{pmatrix}.\]
We therefore have
\[p_2=-sq+f_{i1}'q_1+f_{i1}''q_2=s(f_{k2}^\ell q_1-q)+f_{i1}''q_2.\]
It follows that
\[\begin{pmatrix}p_1\\p_2\end{pmatrix}=\begin{pmatrix}-h_k^\ell&0\\s&f_{i1}''\end{pmatrix}\begin{pmatrix}f_{k2}^\ell q_1-q\\q_2\end{pmatrix},\]
so \eqref{ras'} is right $2$-almost split when there are no arrows $i\to k$ in $Q$, completing the proof.
\end{proof}

We summarise our results in the following theorem, establishing compatibility of different notions of mutation in Frobenius cluster categories.

\begin{thm}
\label{main-mutation-thm}
Let $\frobcat$ be a stably $2$-Calabi--Yau Frobenius category satisfying \ref{C1}--\ref{C2}, and assume there is a cluster-tilting object $T\in\frobcat$ such that $\Endalg{\frobcat}{T}\cong\frjac{Q}{F}{W}$, for a reduced ice quiver with potential $(Q,F,W)$. Assume that $Q$ has no loops, and that the Gabriel quiver of $\Endalg{\frobcat}{\hat{T}}$ has no $2$-cycles for any cluster-tilting object $\hat{T}$ mutation equivalent to $T$. Then
\begin{enumerate}[label=(\roman*)]
\item\label{mainthm-mutcomp} if $T'$ is obtained from $T$ by Iyama--Yoshino mutation, which is well-defined since $Q$ has no loops or $2$-cycles, then $\End_\frobcat(T')^{\op}\cong\frjac{\mu_kQ}{\mu_kF}{\mu_kW}$, where $k$ is the vertex of $Q$ corresponding to the mutated summand, and
\item\label{mainthm-quiv} the Gabriel quiver of $\End_\frobcat(T')^{\op}$ is $\mu_kQ=\mu_k^\FZ Q$, the extended Fomin--Zelevinsky mutation of $Q$ at $k$.
\end{enumerate}
Since $(\mu_kQ,\mu_kF,\mu_kW)$ is reduced and $\mu_kQ$ has no loops, these results may be extended inductively to the entire mutation class of $T$.
\end{thm}

\begin{proof}
Statement \ref{mainthm-mutcomp} is just \Cref{mutation-thm}. Since $\mu_k(Q,F,W)$ is reduced by definition, it follows that $\mu_kQ$ is the Gabriel quiver of $\frjac{\mu_kQ}{\mu_kF}{\mu_kW}\iso\End_\frobcat(T')^{\op}$. Since it has no $2$-cycles, it coincides with $\mu_k^\FZ Q$ by \Cref{p:no-2-cycles}.
\end{proof}

To apply this theorem, we need ways of checking that loops and $2$-cycles do not appear in the relevant quivers. In the case of Hom-finite Frobenius cluster categories, this condition is automatic.

\begin{prop}
\label{p:cluster-quivers}
Let $\frobcat$ be a Hom-finite Frobenius cluster category, and $T\in\frobcat$ a cluster-tilting object. Then the Gabriel quiver $Q$ of $\Endalg{\frobcat}{T}$ has no loops or $2$-cycles.\footnote{This proposition is unfortunately false as stated; see Appendix~\ref{s:corrigendum} for several possible remedies. We have left the statement and `proof' as it originally appeared in the published version of the article.}
\end{prop}
\begin{proof}
Let $A=\Endalg{\frobcat}{T}$. Since $\frobcat$ is a Frobenius cluster category, $\gldim{A}\leq 3$, and since $\frobcat$ is Hom-finite, $A$ is a finite-dimensional algebra. Thus by \cite{igusanoloops,lenzingnilpotente}, $Q$ has no loops.

By \cite[Prop.~3.11]{geissrigid} (which uses again that $A$ is a finite-dimensional algebra of finite global dimension), to show that $Q$ has no $2$-cycles it is enough to show that $\Ext^2_A(S,S)=0$ for any simple $A$-module $S$.

So let $i\in Q_0$, and let $S_i$ be the corresponding simple module. By \Cref{assumption-knitting}, we have access to the sequences from assumption \ref{O}. Consider the sequence ending at $T_i$, and apply $\Hom_\frobcat(T,-)$; independent of whether $i$ is mutable or frozen, this gives us an exact sequence
\begin{equation}
\label{eq:cluster-quivers}
\begin{tikzcd}[column sep=20pt]
\Hom_\frobcat(T,U_{i1})\arrow{r}&\Hom_\frobcat(T,U_{i0})\arrow{r}&\catrad_\frobcat(T,T_i)\arrow{r}&0,
\end{tikzcd}
\end{equation}
with the leftmost two terms being projective $A$-modules since $U_{i1},U_{i0}\in\add{T}$.

Since $\frobcat$ is Hom-finite, it satisfies \ref{C2}, and so $S_i=\Hom_\frobcat(T,T_i)/\catrad_\frobcat(T,T_i)$. Combining this fact with the exact sequence \eqref{eq:cluster-quivers} provides the beginning
\[\begin{tikzcd}
P_2\arrow{r}&P_1\arrow{r}&P_0\arrow{r}&S_i\arrow{r}&0
\end{tikzcd}\]
of a projective resolution of $S_i$ with $P_2=\Hom_\frobcat(T,U_{i1})$. Now, recalling the definition of $U_{i1}$ from assumption \ref{O}, and using that $Q$ has no loops, we see that $T_i$ is not a summand of $U_{i1}$, and hence $\Hom_A(P_2,S_i)=0$. It follows that $\Ext^2_A(S_i,S_i)=0$, as required.
\end{proof}

Consider again the Frobenius cluster categories from \Cref{applications}\ref{eg-BIRS}--\ref{eg-P}. Since these categories are Hom-finite, we can combine \Cref{main-mutation-thm} and \Cref{p:cluster-quivers}\footnote{The corrected Proposition~\ref{p:no-2-cycles-induction} should be applied instead at this point. Alternatively, in the case of Example~\ref{applications}(i), Proposition~\ref{p:subobject-closed} may be used.} to see that the endomorphism algebra of any cluster-tilting object within the mutation class of those referred to in \Cref{applications} has endomorphism algebra isomorphic to a frozen Jacobian algebra, and that mutation of cluster-tilting objects commutes with extended Fomin--Zelevinsky mutation of quivers within these classes.

The argument above does not apply to the Grassmannian cluster categories $\CM(B_{k,n})$ of \Cref{applications}\ref{eg-JKS}, since these are Hom-infinite. However, we can replace \Cref{p:cluster-quivers} with the following argument, and then apply \Cref{main-mutation-thm} as before.

\begin{prop}
\label{p:Grassmannians}
Let $\CM(B_{k,n})$ be the Grassmannian cluster category \cite{jensencategorification} for the Grassmannian $G_k^n$ with $n\geq 3$, as in \Cref{applications}\ref{eg-JKS}, and choose a cluster-tilting object $T\in\frobcat$. Then the Gabriel quiver of $\Endalg{\frobcat}{T}$ has no loops or $2$-cycles.
\end{prop}
\begin{proof}
To simplify the notation, we abbreviate $B_{k,n}$ to $B$. By \cite[Thm.~4.5]{jensencategorification}, there is an exact functor $\pi\colon\CM(B)\to\Sub{Q_k}$, which is a quotient by the ideal generated by an indecomposable projective $B$-module $P_n$. Here $\Sub{Q_k}$ denotes the exact category of submodules of an injective module $Q_k$ for the preprojective algebra of type $\mathsf{A}_{n-1}$, see \cite[\S3]{geisspartial}, and is a Hom-finite Frobenius cluster category \cite[Eg.~3.11]{presslandinternally} (in fact, it is even one of the categories $\cat_w$ considered in \cite{buancluster}; cf.\ \cite[Lem.~17.2]{geisskacmoody}).

As such, $\pi T$ is a cluster-tilting object in $\Sub{Q_k}$, and $\Endalg{\Sub{Q_k}}{\pi T}$ is obtained from $\Endalg{B}{T}$ as the quotient by an idempotent (that given by projection onto the summand $P_n$). Thus the Gabriel quiver of $\Endalg{\Sub{Q_k}}{\pi_T}$ has no loops or $2$-cycles by \Cref{p:cluster-quivers}\footnote{The corrected Proposition~\ref{p:subobject-closed} should be applied instead at this point.}. It follows that any loops or $2$-cycles in the Gabriel quiver of $\Endalg{B}{T}$ must be incident with the vertex corresponding to $P_n$.

However, because of the cyclic symmetry of the algebra $B$, the same argument applies when replacing $P_n$ by one of the $n-1$ other indecomposable projective $B$-modules, giving another quotient functor $\pi'\colon\CM(B)\to\Sub{Q_k}$ (typically with $\pi'T\not\iso\pi T$). Since $n\geq3$, we may apply the above argument to two more of these quotient functors, and thus also rule out any loops or $2$-cycles in the quiver of $\Endalg{B}{T}$ incident with the vertex corresponding to $P_n$.
\end{proof}

We note for completeness that in the one remaining case of $G_1^2$, the projective line, the algebra $B=B_{1,2}$ is the complete path algebra of a $2$-cycle, and the Grassmannian cluster category is $\proj{B}$, which has the unique cluster-tilting object $B$, with no mutable summands.

\section*{Acknowledgements}
We would like to thank Bethany Marsh for pointing out during her lecture series at ICRA 2016 in Syracuse that compatibility of mutations, in the sense of \Cref{main-mutation-thm}, had not yet been established for the Grassmannian cluster categories of \cite{jensencategorification}, and also for comments on an early version of the corrigendum.

\appendix
\section{Corrigendum}
\label{s:corrigendum}
After the publication of this article, we discovered an error in the proof of Proposition~\ref{p:cluster-quivers}. This proposition states that if $T$ is a cluster-tilting object in a Hom-finite Frobenius cluster category $\frobcat$, then the quiver of $A=\Endalg{\frobcat}{T}$ has no loops or $2$-cycles. The part of the statement concerning $2$-cycles is, however, false, and the error in the claimed proof is the use of Lemma~\ref{assumption-knitting}, which applies only when $A$ is isomorphic to a frozen Jacobian algebra. Examples in which $2$-cycles appear may be constructed in the following way. Given any finite-dimensional algebra $A$ of global dimension at most $3$, the category $\proj{A}$ is a Hom-finite Frobenius cluster category (albeit a highly degenerate one that does not categorify an interesting cluster algebra, since its stable category is zero). Up to additive equivalence, the unique cluster-tilting object of $\proj{A}$ is $A$ itself, with endomorphism algebra $\Endalg{A}{A}\isoto A$, and it is not hard to arrange that $A$ is not a frozen Jacobian algebra, or that its quiver has $2$-cycles. Indeed, an explicit example is given by
\[A=k\bigg(\begin{tikzcd}1\arrow[bend left]{r}{\alpha}&2\arrow[bend left]{l}{\beta}\end{tikzcd}\bigg)/\langle \alpha\beta\rangle\]

The easiest way to correct the statement in a way that still permits its application in the proof of Proposition~\ref{p:Grassmannians} is to add a further assumption, as in the following version, which is essentially due to Buan, Iyama, Reiten and Scott \cite{buancluster}.

\begin{prop}
\label{p:subobject-closed}
Let $\frobcat$ be a Hom-finite Frobenius cluster category equivalent as an exact category to a full subcategory, closed under subobjects, of an abelian category. Then if $T\in\frobcat$ is a cluster-tilting object, the quiver of $A=\Endalg{\frobcat}{T}$ has no loops or $2$-cycles.
\end{prop}
\begin{proof}
Hom-finiteness of $\frobcat$ implies that $A$ is finite dimensional, and it is part of the definition of a Frobenius cluster category that $A$ has finite global dimension. Thus we may obtain the result exactly as in the proof of \cite[Prop.~II.1.11(b)]{buancluster}.
\end{proof}

In the proof of Proposition~\ref{p:Grassmannians}, the relevant Frobenius cluster category is the category $\Sub{Q_k}$ of submodules of direct sums of copies of an injective module $Q_k$ for a preprojective algebra $\Pi$, which is a subobject-closed subcategory of the abelian category $\fgmod{\Pi}$, and so Proposition~\ref{p:subobject-closed} may be applied in place of Proposition~\ref{p:cluster-quivers}. We note that another proof of Proposition~\ref{p:Grassmannians}, using essentially the same strategy, is given in \cite[Prop.~4.2]{jensencategorification2}.

Proposition~\ref{p:subobject-closed} also applies to some of the more general Frobenius cluster categories equivalent to the category $\GP(B)$ of Gorenstein projective modules over an Iwanaga--Gorenstein algebra. By \cite[Thm.~2.7]{iyamafrobenius}, any Hom-finite Frobenius cluster category with finitely many isoclasses of indecomposable projectives is equivalent to such a category, and the Gorenstein dimension of $B$ is bounded above by $3$ (see also \cite[Cor.~3.10]{presslandinternally} for the result in this language). In practice, the Gorenstein dimension of $B$ is often strictly smaller than $3$---if it is either $0$ or $1$ then $\GP(B)$ is a subobject-closed subcategory of $\fgmod{B}$, and so Proposition~\ref{p:subobject-closed} once again applies.

Under the original assumptions of Proposition~\ref{p:cluster-quivers}, we can at least rule out the existence of loops, and of $2$-cycles in the quiver of the stable endomorphism algebra of a cluster-tilting object.

\begin{prop}
\label{p:stable-no-2-cycles}
Let $\frobcat$ be a Hom-finite Frobenius cluster category, and $T\in\frobcat$ a cluster-tilting object. Then the quiver of $A=\Endalg{\frobcat}{T}$ has no loops, and the quiver of the stable endomorphism algebra $\stabEndalg{\frobcat}{T}$ has no loops or $2$-cycles.
\end{prop}
\begin{proof}
As in the setting of Proposition~\ref{p:subobject-closed}, $A$ is a finite dimensional algebra of finite global dimension, so it has no loops by the no loops theorem \cite{lenzingnilpotente,igusanoloops}. In \cite[Prop.~3.11]{geissrigid} (which is also used in the proof of \cite[Prop.~II.1.11(b)]{buancluster} cited above), Geiß, Leclerc and Schröer show that the quiver of $A$ has no $2$-cycles provided $\Ext^2_A(S,S)=0$ for any simple $A$-module $S$. In fact, their proof is `local', and shows that there is no $2$-cycle between vertices $i$ and $j$ provided $\Ext^2_A(S_i,S_i)=0=\Ext^2_A(S_j,S_j)$, where $S_i$ and $S_j$ are the simple $A$-modules supported on these vertices.

If $i$ is a vertex corresponding to a non-projective summand of $T$, then we have
\[\Ext^2_A(S_i,S_i)=\Kdual\Ext^1_A(S_i,S_i)=0,\]
the first equality by \cite[\S5.4]{kellerclustertilted} (see also \cite[Thm.~3.4]{presslandinternally}), and the second since there is no loop at $i$. Hence by \cite[Prop.~3.11]{geissrigid} any $2$-cycle in the quiver of $A$ must pass through a vertex corresponding to a projective summand of $T$. Since the quiver of $\stabEndalg{\frobcat}{T}$ is obtained from that of $A$ by deleting these vertices and their incident arrows, it has no $2$-cycles.
\end{proof}

When making connections to cluster theory, we give the quiver of $A$ the structure of an ice quiver by declaring those vertices corresponding to projective summands of $T$ to be frozen. Since arrows between frozen vertices play no role in constructing a cluster algebra from this ice quiver, for applications to cluster algebras we need only rule out $2$-cycles in the quiver passing through at least one mutable vertex (cf.~\cite[\S II.1]{buancluster}). We do not currently know whether the assumptions of Proposition~\ref{p:stable-no-2-cycles} and Proposition~\ref{p:cluster-quivers} are sufficient for this purpose, and indeed our counterexamples to the original statement from only exhibit $2$-cycles between frozen vertices.

Another purpose of Proposition~\ref{p:cluster-quivers} was to provide sufficient conditions under which Theorem~\ref{main-mutation-thm} could be applied, this theorem explaining when mutations of cluster-tilting objects in a Frobenius cluster category are compatible with mutations of frozen Jacobian algebras, and with (extended) Fomin--Zelevinsky mutations of quivers. In this context, we are given a Frobenius cluster category $\frobcat$ containing a cluster-tilting object $T$ for which $\Endalg{\frobcat}{T}$ is isomorphic to a frozen Jacobian algebra. To be in the setting of Theorem~\ref{main-mutation-thm}, we need to rule out $2$-cycles in the quivers of $\Endalg{\frobcat}{\hat{T}}$ for cluster-tilting objects $\hat{T}$ mutation equivalent to $T$. The next proposition shows that this follows from the other assumptions of the theorem when $\frobcat$ is Hom-finite.

\begin{prop}
\label{p:no-2-cycles-induction}
Let $\frobcat$ be a Hom-finite Frobenius cluster category, and let $T\in\frobcat$ be a cluster-tilting object such that $\Endalg{\frobcat}{T}\iso\frjac{Q}{F}{W}$ for a reduced ice quiver with potential $(Q,F,W)$. If $Q$ has no loops or $2$-cycles, then the quiver of $\Endalg{\frobcat}{\hat{T}}$ has no loops or $2$-cycles for any $\hat{T}$ mutation equivalent to $T$.
\end{prop}
\begin{proof}
The quiver of $\Endalg{\frobcat}{\hat{T}}$ has no loops by Proposition~\ref{p:stable-no-2-cycles}, the assumptions of which are weaker than those of the present statement, and so it is sufficient for us to rule out $2$-cycles.

The quiver $Q$, which is the quiver of $\Endalg{\frobcat}{T}$ since $(Q,F,W)$ is reduced, has no $2$-cycles by assumption. Let $T'$ be a cluster-tilting object obtained from $T$ by a single mutation, say at the summand corresponding to vertex $k\in Q_0$. Writing $(Q',F',W')=\mu_k(Q,F,W)$ for the mutation of $(Q,F,W)$ at this vertex, it follows from Theorem~\ref{mutation-thm} that there is an isomorphism
\[A'=\Endalg{\frobcat}{T'}\iso\frjac{Q'}{F'}{W'},\]
and so in particular $A'$ is isomorphic to a frozen Jacobian algebra. As a result, the argument given in Proposition~\ref{p:cluster-quivers} does in fact apply in this case, and we summarise it here. Since $A'$ is a frozen Jacobian algebra, there is an exact sequence
\begin{equation}
\label{eq:simp-proj-res}
\bigoplus_{\stackrel{\beta\in (Q')_1^\mut}{\head{\beta}=i}}P_{\tail{\beta}}\to\bigoplus_{\stackrel{\alpha\in Q'_1}{\tail{\alpha}=i}}P_{\head{\alpha}}\to P_i\to S_i\to 0
\end{equation}
for any $i\in Q'_0$, where $S_i$ is the simple module at $i$ and $P_j=A'e_j$ is the indecomposable projective module with top $S_j$. We recall also the notation $(Q')_1^\mut$ for the set of unfrozen arrows of $Q'$. Since $(Q',F',W')$ is reduced by the definition of mutation, $Q'$ is the quiver of $A'$, and thus has no loops by Proposition~\ref{p:stable-no-2-cycles}. It follows that $\tail{\beta}\ne i$ whenever $\head{\beta}=i$, and so
\[\Hom_{A'}\Bigg(\bigoplus_{\stackrel{\beta\in (Q')_1^\mut}{\head{\beta}=i}}P_{\tail{\beta}},S_i\Bigg)=0.\]
Since $\Ext^2_{A'}(S_i,S_i)$ is a subquotient of this space, \eqref{eq:simp-proj-res} being the start of a projective resolution of $S_i$, we also have $\Ext^2_{A'}(S_i,S_i)=0$. It follows that the quiver $Q'$ of $A'$ has no $2$-cycles by \cite[Prop.~3.11]{geissrigid}. Applying the argument inductively, we extend this conclusion to the entire mutation class of $T$.
\end{proof}

\begin{rem}
We have taken this opportunity to also correct Definitions~\ref{d:potential} and \ref{d:jac-alg}, in which the ideal generated by commutators was erroneously used in place of the vector subspace spanned by commutators.

Secondly, Chang and Zhang point out \cite[Rem.~2.9]{changquivers} that the ice quivers with potential associated to $(k,n)$-Postnikov diagrams \cite{baurdimer} are not rigid in the sense of Definition~\ref{rigid}. Since this family should be particularly well-behaved (and these ice quivers with potential are non-degenerate by Proposition~\ref{p:Grassmannians} and \cite[Prop.~4.2]{jensencategorification2}), our interpretation of their remark is that the definition of rigidity given here is too strong, and should be replaced by a better notion. However, at this time we do not have an alternative suggestion.
\end{rem}

\defbibheading{bibliography}[\refname]{\section*{#1}}\printbibliography
\end{document}